\documentclass[fleqn,final]{zzz}
\usepackage{mathrsfs}
\usepackage{color}
\usepackage{tikz}
\usepackage{float}
\usepackage{graphicx}
\usepackage{rotating}
\usepackage[pdftex]{hyperref}
\usepackage{enumerate}
\usepackage{amsmath}
\usepackage{stmaryrd}

\hypersetup{
 colorlinks = true,
 urlcolor = blue,
 linkcolor = red,
 citecolor = green
}

\newcommand\Qbinom[3]{\genfrac{[}{]}{0pt}{}{#1}{#2}_{#3}}
\newcommand\qbinom[2]{\Qbinom{#1}{#2}{q}}

\newcommand{\hyp}[5]{\,\mbox{}_{#1}F_{#2}\!\left(
 \genfrac{}{}{0pt}{}{#3}{#4};#5\right)}
\newcommand{\qhyp}[5]{\,\mbox{}_{#1}\phi_{#2}\!\left(
\genfrac{}{}{0pt}{}{#3}{#4};#5\right)}
\newcommand{\Whyp}[5]{\,\mbox{}_{#1}W_{#2}\!\left({#3};{#4};{#5}\right)}
\newcommand{\qphyp}[6]{\,\sideset{_{#1}^{\phantom{\mid}}}{_{#2}^{#3}}
{\mathop{\phi}}\!\left(\genfrac{}{}{0pt}{}{#4}{#5};#6\right)}
\newcommand{\nqphyp}[3]{\,\sideset{_{#1}^{\phantom{\mid}}}{_{#2}^{#3}}{\mathop{\phi}}}
\newcommand{\nlqphyp}[3]{\,\sideset{_{#1}}{_{#2}^{#3}}{\mathop{\phi}}}
\newcommand{\rphis}[2]{{}_{#1\vphantom{#2}}\phi_{#2\vphantom{#1}}}
\newcommand{\rphisx}[4]{\rphis{#1}{#2}\left( \begin{array}{c} #3 \end{array};q,#4\right)}
\newcommand{\II}[1]{\overset{\raisebox{0.5ex}{$#1$}}{\raisebox{-0.55cm}
{\resizebox{0.3cm}{!}{\scalebox{0.4}[1.0]{\mbox{$\mathbb{I}$}}}}}}
\newcommand{\quhyp}[6]{\nlqphyp{#1}{#2}{#3}\left(\genfrac{}{}{0pt}{}{#4}{#5};#6\right)}

\newcommand{\CC}{{{\mathbb C}}}
\newcommand{\CCast}{{{\mathbb C}^\ast}}
\newcommand{\CCdag}{{{\mathbb C}^\dag}}
\newcommand{\SSS}{{\mathcal S}}

\def\cprime{$'$}

\makeatletter 
\newcommand{\Vast}{\bBigg@{4.9}} 
\newcommand{\vast}{\bBigg@{4}} 
\makeatother

\newtheorem{thm}[lemma]{Theorem}
\newtheorem{cor}[lemma]{Corollary}
\newtheorem{rem}[lemma]{Remark}

\newtheorem{defn}[lemma]{Definition}

\makeatletter
\def\eqnarray{\stepcounter{equation}\let\@currentlabel=\theequation
\global\@eqnswtrue
\tabskip\@centering\let\\=\@eqncr
$$\halign to \displaywidth\bgroup\hfil\global\@eqcnt\z@
 $\displaystyle\tabskip\z@{##}$&\global\@eqcnt\@ne
 \hfil$\displaystyle{{}##{}}$\hfil
 &\global\@eqcnt\tw@ $\displaystyle{##}$\hfil
 \tabskip\@centering&\llap{##}\tabskip\z@\cr}

\def\endeqnarray{\@@eqncr\egroup
 \global\advance\c@equation\m@ne$$\global\@ignoretrue}

\def\@yeqncr{\@ifnextchar [{\@xeqncr}{\@xeqncr[5pt]}}
\makeatother

\newcommand{\dd}{{\mathrm d}}
\newcommand{\expe}{{\mathrm e}}

\newcommand{\om}{{\boldsymbol\omega}}

\newcommand{\C}{\mathbb{C}}
\newcommand{\Z}{\mathbb{Z}}
\newcommand{\N}{\mathbb{N}}

\numberwithin{equation}{section}
\numberwithin{equation}{section}
\numberwithin{corollary}{section}
\numberwithin{remark}{section}
\numberwithin{theorem}{section}
\numberwithin{lemma}{section}

\begin{document}

\renewcommand{\PaperNumber}{***}

\FirstPageHeading

\ArticleName{Nonterminating transformations and summations\\associated with some $q$-Mellin--Barnes integrals}
\ShortArticleName{Nonterminating transformations and summations for some $q$-Mellin--Barnes integrals}

\Author{Howard S. Cohl\,$^\dag\!\!\ $ and 
Roberto S. Costas-Santos$\,^{\S}$}

\AuthorNameForHeading{H.~S.~Cohl and R.~S.~Costas-Santos}

\Address{$^\dag$ Applied and Computational
Mathematics Division, National Institute of Standards
and Technology, Mission Viejo, CA 92694, USA
\URLaddressD{
\href{http://www.nist.gov/itl/math/msg/howard-s-cohl.cfm}
{http://www.nist.gov/itl/math/msg/howard-s-cohl.cfm}
}
} 
\EmailD{howard.cohl@nist.gov} 

\Address{$^\S$ Dpto. de F\'isica y Matem\'{a}ticas,
Universidad de Alcal\'{a},
c.p. 28871, Alcal\'{a} de Henares, Spain} 
\URLaddressD{
\href{http://www.rscosan.com}
{http://www.rscosan.com}
}
\EmailD{rscosa@gmail.com} 


\ArticleDates{Received \today in final form ????; Published online ????}

\Abstract{
In many cases one may encounter an integral which is of
$q$-Mellin--Barnes type. These integrals are easily evaluated
using theorems which have a long history dating back to
Slater, Askey, Gasper, Rahman and others.
We derive some interesting $q$-Mellin--Barnes integrals
and using them we derive transformation
and summation formulas for nonterminating basic hypergeometric 
functions. The cases which we treat include ratios 
of theta functions, the Askey--Wilson moments, nonterminating well-poised
${}_3\phi_2$, nonterminating very-well-poised ${}_5W_4$, ${}_8W_7$, products of two nonterminating ${}_2\phi_1$'s, square of a nonterminating well-poised ${}_2\phi_1$, and nonterminating ${}_{12}W_{11}$ and ${}_{10}W_9$.
}


\Classification{33D15; 33D60}

\section{Preliminaries}%


There have existed in the past some very
important $q$-Mellin--Barnes integrals. Some
important examples include those given
by Askey--Wilson \cite[{(2.1)}]{AskeyWilson85}, 
Nassrallah--Rahman {\cite[(6.3.9)]{GaspRah}}
as well as those given
in Askey--Roy \cite[(2.8)]{AskeyRoy86}
and in Gasper \cite[(1.8)]{Gasper1989}.
In this paper we take advantage
of the powerful methods following the
pioneering work of Bailey
\cite[Chapter 8]{Bailey64}, 
and his student
Slater 
which were fully {recapitulated}
by Gasper \& Rahman 
in \cite[Chapter 4]{GaspRah}.
We are able to use well-known
formulas for certain highly 
symmetric basic hypergeometric functions
to obtain new $q$-Mellin--Barnes integrals and from them derive a new class of transformation and summation
formulas.

\medskip
We adopt the following set 
notations: $\mathbb N_0:=\{0\}\cup\N=\{0, 1, 2, \ldots\}$, and we 
use the sets $\mathbb Z$, $\mathbb R$, $\mathbb C$ which represent 
the integers, real numbers and 
complex numbers respectively, $\CCast:=\CC\setminus\{0\}$, and 
{$\CCdag:=\{z\in\CCast: |z|<1\}$.}
We also adopt the following notation and conventions.
Given a set ${\bf a}:=\{a_1,\ldots,a_A\}$, for $A\in\N$, then we define ${\bf a}_{[k]}:={\bf a}\setminus\{a_k\}$, $1\le k\le A$,
$b\,{\bf a}:=\{b\, a_1, b\, a_2, \ldots, b\, a_A\}$,
${\bf a}+b:=\{a_1+b,a_2+b,\ldots,a_A+b\}$,
where
$b,a_1,\ldots,a_A\in\mathbb C$.

\begin{rem} \label{rem:2.8}
Observe in the following discussion we will often be referring
to a collection of constants $a,b,c,d,e,f$. In such cases, which will be
clear from context, then the constant $e$ should not be confused
with Euler's number $\expe$, the base of the natural logarithm, i.e.,
$\log\expe=1$. Observe the different
(roman) typography for Euler's number.
\end{rem}

We assume that the empty sum 
vanishes and the empty product 
is unity.
{
We will also adopt the following symmetric sum notation.
\begin{defn}
For some function $f(a_1,\ldots,a_n;{\bf b})$,
where ${\bf b}$ is some set of parameters.
Then 
\begin{equation}
\II{a_1;a_2,\ldots,a_n}\!\!\!
f(a_1,a_2,\ldots,a_n;{\bf b})
:=f(a_1,a_2,\ldots,a_n;{\bf b})+
{\mathrm{idem}}(a_1;a_2,\ldots,a_n),
\end{equation}
where ``\,${\mathrm{idem}}(a_1;a_2,\ldots,a_n)$'' after an 
expression stands for the sum of the $n-1$ expressions 
obtained from the preceding expression by interchanging $a_1$ 
with each $a_k$, $k=2,3,\ldots,n$.
\end{defn}
}
\begin{defn} \label{def:2.1}
We adopt the following conventions for succinctly 
writing elements of {sets}. To indicate sequential positive and negative 
elements, we write
\[
\pm a:=\{a,-a\}.
\]
\noindent We also adopt an analogous notation
\[
\expe^{\pm i\theta}:=\{\expe^{i\theta},\expe^{-i\theta}\}.
\]
\noindent In the same vein, consider the numbers $f_s\in\mathbb C$ 
with $s\in{\mathcal S}\subset \N$,
with ${\mathcal S}$ finite.
Then, the notation
$\{f_s\}$
represents the set of all complex numbers $f_s$ such that 
$s\in\SSS$.
Furthermore, consider some $p\in\SSS$, then the notation
$\{f_s\}_{s\ne p}$ represents the sequence of all complex numbers
$f_s$ such that $s\in\SSS\!\setminus\!\{p\}$.
\end{defn}

Consider $q\in\CCdag$, $n\in\mathbb N_0$.
Define the sets 
\begin{eqnarray}
&&\hspace{-7.5cm}\Omega_q^n:=\{q^{-k}: k\in\mathbb N_0,~0\le k\le n-1\},\\
&&\hspace{-7.5cm}\Omega_q:=\Omega_q^\infty=\{q^{-k}:k\in\mathbb N_0\}, \\
&&\hspace{-7.5cm}\Upsilon_q:=\{q^k:k\in\Z\}.
\end{eqnarray}
In order to obtain our derived identities, we rely on properties 
of the {$q$-shifted factorial $(a;q)_n$}. {We refer to 
$(a;q)_n$ as a $q$-shifted factorial (it is also referred to as a $q$-Pochhammer symbol).}
For any $n\in \mathbb N_0$, $a,{b}, q \in \mathbb C$, 
the 
$q$-shifted factorial is defined as
\begin{eqnarray}
&&\hspace{-7.0cm}\label{poch.id:1} 
(a;q)_n:=(1-a)(1-aq)\cdots(1-aq^{n-1}).
\end{eqnarray}
One may also define
\begin{eqnarray}
&&\hspace{-9.6cm}(a;q)_\infty:=\prod_{n=0}^\infty (1-aq^{n}), \label{infPochdefn}
\end{eqnarray}
where $|q|<1$. 
Furthermore, one has the following
identities
\begin{eqnarray}
&&\hspace{-9.3cm}\label{sq-1}
(a^2;q)_\infty=(\pm a,\pm q^\frac12 a;q)_\infty\label{a2q}\\
&&\hspace{-9.3cm}\label{sq-2}
(a^2;q^2)_\infty=(\pm a;q)_\infty.\label{a2q2}
\end{eqnarray}
One also has
\begin{eqnarray}
&&\hspace{-8.5cm}(q^{-n}a;q)_\infty=(q^{-n}a;q)_n(a;q)_\infty.
\label{infPoch2}
\end{eqnarray}
One also has the definition
of the $q$-gamma function, namely
\cite[(1.9.1)]{Koekoeketal}
\begin{equation}
\Gamma_q(x):=\frac{(q;q)_\infty}
{(1-q)^{x-1}(q^x;q)_\infty},
\label{qgam}
\end{equation}
and also the gamma function $\Gamma:\C\setminus-\N_0\to\C$ defined
in \cite[(5.2.1)]{NIST:DLMF}.
Note that \cite[p.~13]{Koekoeketal}
\begin{equation}
\lim_{q\to1^{-}}\Gamma_q(x)=\Gamma(x).
\label{limqgam}
\end{equation}
We will also use the {following} notational product conventions, $a_k\in\CC$, $k\in\N$, $b\in\C\cup\{\infty\}$,
\begin{eqnarray}
&&\hspace{-7.7cm}(a_1,\ldots,a_k;q)_b:=(a_1;q)_b\cdots(a_k;q)_b,\\
&&\hspace{-7.7cm}\Gamma_q(a_1,\ldots,a_k):=\Gamma_q(a_1)\cdots\Gamma_q(a_k),\\
&&\hspace{-7.7cm}\Gamma(a_1,\ldots,a_k):=\Gamma(a_1)\cdots\Gamma(a_k).
\end{eqnarray}

The basic hypergeometric series, 
which we 
will often use, is defined for
{$z\in \mathbb C$}, 
$q\in\CCdag$, 
$s\in\mathbb N_0$, $r\in\mathbb N_0\cup\{-1\}$,
$b_j\not\in\Omega_q$, $j=1,\ldots,s$, as
\cite[(1.10.1)]{Koekoeketal}
\begin{equation}
\qhyp{r+1}{s}{a_1,\ldots,a_{r+1}}
{b_1,\ldots,b_s}
{q,z}
:=\sum_{k=0}^\infty
\frac{(a_1,\ldots,a_{r+1};q)_k}
{(q,b_1,\ldots,b_s;q)_k}
\left((-1)^kq^{\binom k2}\right)^{s-r}
z^k.
\label{2.11}
\end{equation}
\noindent {For $s>r$, ${}_{r+1}\phi_s$ is an entire
function of $z$, for $s=r$ then 
${}_{r+1}\phi_s$ is convergent for $|z|<1$, and for $s<r$ the series
is divergent {unless it is terminating}.}
Note that when we refer to a basic hypergeometric
function with {\it arbitrary argument} $z$, we simply mean that
the argument does not necessarily depend on the other parameters, namely the $a_j$'s, $b_j$'s nor $q$. However, for the arbitrary argument $z$, it very-well
may be that the domain of the argument
is restricted, such as for $|z|<1$.

\medskip 
We will use the following notation 
$\nqphyp{r+1}{s}{m}$, $m\in\mathbb Z$
(originally due to van de Bult \& Rains
\cite[p.~4]{vandeBultRains09}), 
for basic hypergeometric series {when some parameter entries are equal to zero}.
Consider $p\in\mathbb N_0$. Then define
\begin{equation}\label{topzero} 
\nlqphyp{r+1}{s}{-p}\!
\left(\!\begin{array}{c}a_1,\ldots,a_{r+1} \\
b_1,\ldots,b_s\end{array};q,z
\right)
:=
\rphisx{r+p+1}{s}{a_1,a_2,\ldots,a_{r+1},\overbrace{0,\ldots,0}^{p} \\ b_1,b_2,\ldots,b_s}{z},
\end{equation}
\begin{equation}\label{botzero}
\nlqphyp{r+1}{s}{\,p}\!
\left(\!\begin{array}{c}a_1,\ldots,a_{r+1} \\
b_1,\ldots,b_s\end{array};q,z
\right)
:=
\rphisx{r+1}{s+p}{a_1,a_2,\ldots,a_{r+1} \\ b_1,b_2,\ldots,b_s, \underbrace{0,\ldots,0}_p}{z},
\end{equation}
where $b_1,\ldots,b_s\not
\in\Omega_q\cup\{0\}$, and
$
\nlqphyp{r+1}{s}{0}
={}_{r+1}\phi_{s}
.$
The nonterminating basic hypergeometric series 
$\nlqphyp{r+1}{s}{m}
({\bf a};{\bf b};q,z)$, ${\bf a}:=\{a_1,\ldots,a_{r+1}\}$,
${\bf b}:=\{b_1,\ldots,b_s\}$, is well-defined for $s-r+m\ge 0$. In particular 
$\nlqphyp{r+1}{s}{m}$
is an entire function of $z$ for $s-r+m>0$, convergent for $|z|<1$ for $s-r+m=0$ and divergent if $s-r+m<0$ {unless it is terminating}.
Note that we will move interchangeably between the
van de Bult \& Rains notation and the alternative
notation with vanishing numerator and denominator parameters
which are used on the right-hand sides of \eqref{topzero} and \eqref{botzero}.

The geometric series is given
by \cite{Andrews98}
\begin{equation}
\sum_{n=0}^\infty z^n=\frac{1}{1-z},
\label{geom}
\end{equation}
provided $|z|<1$.
The $q$-binomial theorem is given by
\cite[(1.11.1)]{Koekoeketal}
\begin{equation}
\qhyp10{a}{-}{q,z}=\frac{(az;q)_\infty}
{(z;q)_\infty},
\label{qbinom}
\end{equation}
provided $|z|<1$ for 
convergence of the left-hand side
nonterminating basic hypergeometric
series.

\subsection{The theta function and the partial theta function}
\label{tfptf}
The {\it theta function} $\vartheta(z;q)$ (sometimes referred to as a modified theta function 
\cite[(11.2.1)]{GaspRah})
is defined
by 
Jacobi's triple product identity and is
given by {\cite[(1.6.1)]{GaspRah}} (see also \cite[(2.3)]{Koornwinder2014})
\begin{equation}
\vartheta(z;q):=
(z,q/z;q)_\infty=\frac{1}{(q;q)_\infty}\sum_{n=-\infty}^\infty (-1)^nq^{\binom{n}{2}}z^n.
\label{tfdef}
\end{equation}
where $z\ne 0$. Note that $\vartheta(q^n;q)=0$ if
$n\in\Z$.
We will adopt the product convention
for theta functions for $a_k\in\C$ for $k\in\N$, namely
\begin{eqnarray}
&&\hspace{-8cm}\vartheta(a_1,\ldots,a_k;q):=\vartheta(a_1;q)\cdots\vartheta(a_k;q).\nonumber
\end{eqnarray}
A particular ratio of theta function satisfies the
following useful identity
\begin{equation}
\frac{(a,q/a;q)_\infty}
{(qa,1/a;q)_\infty}=
{\frac{\vartheta(a;q)}{\vartheta(qa;q)}}=-a,
\label{infneg}
\end{equation}
where $a\ne 0$.

\medskip
\noindent The {\it partial theta function} $\Theta(z;q)$, described as such because it only involves
the partial sum contribution for $n\ge 0$ in
\eqref{tfdef}
as opposed to summing over all integers as in 
the theta function, 
is defined
as follows with alternative representations.
\begin{thm}
Let $q\in\CCdag$, $p\in\N$,
$z\in\CCast$, $|z|<1$. Then
\begin{eqnarray}
&&\hspace{-3.8cm}\Theta(z;q):=\frac{1}{(q;q)_\infty}\sum_{n=0}^\infty (-1)^n q^{\binom{n}{2}}z^n=\frac{1}{(q;q)_\infty}\quhyp101{q}{-}{q,z}
\label{pt1}\\
&&\hspace{-2.45cm}=\frac{1}{(q;q)_\infty}
\qphyp{1}{0}{p}{q^{1/p}}{-}
{q^{1/p},(-1)^{p-1}z}\label{pt5}\\[0.05cm]
&&\hspace{-2.45cm}=(z;q)_\infty\qphyp{0}{1}{-2}{-}{z}{q,q}\label{pt2}\\[0.05cm]
&&\hspace{-2.45cm}=\frac{(z;q)_\infty}{(q;q)_\infty}\qhyp01{-}{z}{q,qz}\label{pt3}\\[0.05cm]
&&\hspace{-2.45cm}=\frac{(z;q)_\infty}{(\pm q;q)_\infty}\qhyp43{\pm i\sqrt{z},
\pm i\sqrt{qz}}{-q,\pm z}{q,q}.\label{pt4}
\end{eqnarray}
\end{thm}
\begin{proof}
The representation \eqref{pt1} follows 
the definition of nonterminating
basic hypergeometric series \eqref{2.11}.
The representation \eqref{pt5} follows
from direct substitution using
\eqref{2.11}.
The representations \eqref{pt2}, \eqref{pt3}
follow from \cite[(1.13.8-9)]{Koekoeketal}.
The representation \eqref{pt4} follows
from Andrews \& Warnaar's formula for a product of partial
theta functions \cite[Theorem 1.1]{AndrewsW2007} 
\begin{equation}
\Theta(a;q)\Theta(b;q)=\frac{(a,b;q)_\infty}{(q;q)_\infty}
\qhyp43{\pm\sqrt{ab},\pm\sqrt{\frac{ab}{q}}}{a,b,\frac{ab}{q}}{q,q},
\end{equation}
with the substitutions $(a,b)\mapsto(z,-q)$
and the identity \cite[(20.4.3)]{NIST:DLMF}
\begin{equation}
\Theta(-q;q)=(-q,-q;q)_\infty.
\end{equation}
This completes the proof.
\end{proof}

\subsection{Some theorems involving $q$-Mellin--Barnes integrals}

\medskip
Now we present a result which 
allows one to evaluate integrals 
of products and ratios of
infinite 
$q$-shifted factorials in terms
of sums of non-terminating basic 
hypergeometric functions.
The following result is a special case
($t=1$) of the more general
result which 
appears in \cite[Theorem 2.1]{CohlCostasSantos22}.
Note that we adopt a representation
for the contour integral as in 
\cite[(4.9.3)]{GaspRah}. However,
there are several other alternative integral representations
which can be used (see \cite[\S4.9]{GaspRah}).
\begin{thm} \label{gascard}
Let $q\in\CCdag$, $m\in \mathbb Z$,
$\sigma\in (0,\infty)$,
${\bf a}:=\{a_1,\ldots,a_A\}$, 
${\bf b}:=\{b_1,\ldots,b_B\}$, 
${\bf c}:=\{c_1,\ldots,c_C\}$, 
${\bf d}:=\{d_1,\ldots,d_D\}$ be
sets of {non-zero} complex numbers with cardinality 
$A, B, C, D\in\mathbb N_0$ (not all zero) respectively with
$|c_k|<\sigma$,
$|d_l|<1/\sigma$,
for any 
$a_i, b_j, c_k, d_l\in\CCast$ elements of 
${\bf a}, {\bf b}, {\bf c}, {\bf d}$,
and $z:=\expe^{i\psi}$.
Define the {\it $q$-Mellin--Barnes integral}
\begin{eqnarray}
&&\hspace{-1.7cm}G_{m}:=G_{m}({\bf a},{\bf b},{\bf c},{\bf d};\sigma,q)
:=\frac{(q;q)_\infty}{2\pi}\left(\frac{1}{\sigma}\right)^m\int_{-\pi}^\pi
\frac{({\bf b}\frac{\sigma}{z}, {\bf a}\frac{z}{\sigma};q)_\infty}
{({\bf d}\frac{\sigma}{z}, {\bf c}\frac{z}{\sigma};q)_\infty}
\expe^{im\psi}{\mathrm d}{\psi},
\label{gascardeq}
\end{eqnarray}
such that the integral exists. Then
\begin{equation}
G_{m}({\bf a},{\bf b},{\bf c},{\bf d};\sigma,q)
=G_{-m}({\bf b},{\bf a},{\bf d},{\bf c};\sigma,q),
\label{Gmt0}
\end{equation}
if $|c_k|,|d_l|<\min\{1/\sigma, 
\sigma\}$.
Furthermore, let $d_lc_k\not\in\Omega_q$.
If $D\ge B$, $d_l/d_{l'}\not\in\Omega_q$, $l\ne l'$, then
\begin{eqnarray}
\label{Gmt1}
&&\hspace{-0.6cm}G_{m}\!=\!
\sum_{k=1}^D
\frac{(d_k{\bf a},{\bf b}/d_k;q)_\infty d_k^m}
{(d_k{\bf c},{\bf d}_{[k]}/d_k;q)_\infty}
\nqphyp{B+C}{A+D-1}{C-A}
\left(\begin{array}{c}
d_k{\bf c},q d_k/{\bf b}\\
d_k{\bf a},q d_k/{\bf d}_{[k]}
\end{array}
;q,q^m(qd_k)^{D-B}\frac{b_1\cdots b_B}{d_1\cdots d_D}\right)\!,
\end{eqnarray}
and/or if $C\ge A$, $c_k/c_{k'}\not\in\Omega_q$, $k\ne k'$, then
\begin{eqnarray}
&&\hspace{-0.55cm}G_{m}=
\sum_{k=1}^C\!\frac{(\,c_{k}{\bf b},
{\bf a}/c_{k};q)_\infty c_k^{-m}
}
{(c_k{\bf d},{\bf c}_{[k]}/c_k ;q)_\infty}
\nqphyp{A+D}{{B+C-1}}{{D-B}}
\!\!\left(\begin{array}{c} c_k{\bf d}, q c_{k}/
{\bf a}\\ c_k{\bf b},q c_k/{\bf c}_{[k]} \end{array}\!\!\!;q,
q^{-m}(q c_k)^{C-A}\frac{a_1\cdots a_A}{c_1\cdots c_C}\right)\!,
\label{Gmt2}
\end{eqnarray}
\noindent where the nonterminating basic hypergeometric series 
in \eqref{Gmt1} 
(resp.~\eqref{Gmt2}) 
is entire if $D>B$ 
(resp.~$C>A$), convergent for
$|q^mb_1\cdots b_B|<|d_1\cdots d_D|$ if $D=B$ 
(resp. $|q^{-m}a_1\cdots a_A|<|c_1\cdots c_C|$ if $C=A$), 
and divergent otherwise.
\end{thm}
\begin{proof}
See proof of \cite[Theorem 2.1]{CohlCostasSantos22}.
\end{proof}

{One can convert the integral
in the above theorem to a form
which is more similar to that which
appears in Mellin--Barnes integrals by
replacing the infinite $q$-shifted
factorials with $q$-gamma functions
using \eqref{qgam}. 
}
{
\begin{cor}
Let $q\in\CCdag$, $m\in \mathbb Z$,
${\bf a}:=\{a_1,\ldots,a_A\}$, 
${\bf b}:=\{b_1,\ldots,b_B\}$, 
${\bf c}:=\{c_1,\ldots,c_C\}$, 
${\bf d}:=\{d_1,\ldots,d_D\}$ be
sets of {non-zero} complex numbers with cardinality 
$A, B, C, D\in\mathbb N_0$ (not all zero) respectively,
$\Sigma a_j:=\sum_{j=1}^A a_j$,
$\Sigma b_j:=\sum_{j=1}^B b_j$,
$\Sigma c_j:=\sum_{j=1}^C c_j$,
$\Sigma d_j:=\sum_{j=1}^D d_j$,
and
$|q|^\sigma\in (0,\infty)$,
$|q^{c_k}|,|q^{d_l}|<
\min\{|q|^{\sigma},|q|^{-\sigma}\}$,
$d_l+c_k\not\in-\N_0$,
for any 
$a_i, b_j, c_k, d_l\in\CCast$ elements of 
${\bf a}, {\bf b}, {\bf c}, {\bf d}$.
Define
\begin{eqnarray}
&&\hspace{-1.4cm}I_m:=I_m({\bf a},{\bf b},{\bf c},{\bf d};\sigma;q)\nonumber\\
&&\hspace{-0.85cm}:=\int_{\frac{\pi}{\log q}}^{-\frac{\pi}{\log q}}
\frac{\Gamma_q({\bf d}+\sigma-ix,{\bf c}-\sigma+ix)}
{\Gamma_q({\bf b}+\sigma-ix,{\bf a}-\sigma+ix)}
q^{imx}(1-q)^{(ix-\sigma)(C-D+B-A)}\,\dd x.
\end{eqnarray}
Then if $D\ge B$, $d_l-d_{l'}\not\in-\N_0$,
$l\ne l'$, one has
\begin{eqnarray}
&&\hspace{0.0cm}I_m=\frac{2\pi(1-q)q^{m\sigma}}
{-\log q}
\sum_{k=1}^D
\frac{\Gamma_q(d_k+{\bf c},{\bf d}_{[k]}-d_k)}
{\Gamma_q(d_k+{\bf a},{\bf b}-d_k)}
q^{md_k}(1-q)^{d_k(C-D+B-A)}\nonumber\\
&&\hspace{2.7cm}\times
\!\qphyp{B+C}{A+D-1}{C-A}{q^{d_k+{\bf c}},q^{1+d_k-{\bf b}}}
{q^{d_k+{\bf a}},q^{1+d_k-{\bf d}_{[k]}}}
{q,q^{m+(D-B)(1+d_k)+\Sigma b_j-\Sigma d_j}},
\end{eqnarray}
and if $C\ge A$, $c_k-c_{k'}\not\in-\N_0$,
$k\ne k'$, one has
\begin{eqnarray}
&&\hspace{0.0cm}I_m=\frac{2\pi(1-q)q^{m\sigma}}
{-\log q}
\sum_{k=1}^C
\frac{\Gamma_q(c_k+{\bf d},{\bf c}_{[k]}-c_k)}
{\Gamma_q(c_k+{\bf b},{\bf a}-c_k)}
q^{-mc_k}(1-q)^{-c_k(C-D+B-A)}\nonumber\\
&&\hspace{2.7cm}\times
\!\qphyp{A+D}{B+C-1}{D-B}{q^{c_k+{\bf d}},q^{1+c_k-{\bf a}}}
{q^{c_k+{\bf b}},q^{1+c_k-{\bf c}_{[k]}}}
{q,q^{-m+(C-A)(1+c_k)+\Sigma a_j-\Sigma c_j}}.
\end{eqnarray}
\end{cor}
}
{
\begin{proof}
Using \eqref{Gmt1}, \eqref{Gmt2}, we 
respectively start along the lines
of Askey \& Roy \cite[p.~368]{AskeyRoy86}
and use the map $({\bf a},{\bf b},{\bf c},{\bf d},\sigma,\expe^{i\psi})\mapsto(q^{\bf a},q^{\bf b},q^{\bf c},q^{\bf d},q^\sigma,q^{ix})$.
This completes the proof.
\end{proof}
}
{
Note that 
\begin{equation}
\lim_{q\to1^{-}}\frac{-\log q}{1-q}=1.
\label{qloglim}
\end{equation}
So certainly in the case where all infinite
$q$-shifted factorials are composed
of parameters which do not have leading 
negative factors, we can convert the 
integral in Theorem \ref{gascaru} to 
one which resembles a Mellin--Barnes 
integral in the $q\to 1^{-}$ limit
\eqref{limqgam}. 
It is this reason that we refer to 
these integrals as $q$-Mellin--Barnes integrals.
It is also clear that there are situations where the 
$q\to 1^{-}$ limit either vanishes or
is is perhaps not well-defined. This is 
a technicality may or may not be easily 
addressed.
}

\medskip
{Now consider the situation where
$D=B$ and $C=A$. This produces the following
result.}
{
\begin{cor}
Let $q\in\CCdag$, $m\in \mathbb Z$,
${\bf a}:=\{a_1,\ldots,a_A\}$, 
${\bf b}:=\{b_1,\ldots,b_B\}$, 
${\bf c}:=\{c_1,\ldots,c_A\}$, 
${\bf d}:=\{d_1,\ldots,d_B\}$ be
sets of {non-zero} complex numbers with cardinality 
$A, B, C, D\in\mathbb N_0$ (not all zero) respectively,
$\Sigma a_j:=\sum_{j=1}^A a_j$,
$\Sigma b_j:=\sum_{j=1}^B b_j$,
$\Sigma c_j:=\sum_{j=1}^A c_j$,
$\Sigma d_j:=\sum_{j=1}^B d_j$,
and
$|q|^\sigma\in (0,\infty)$,
$|q^{c_k}|,|q^{d_l}|<
\min\{|q|^{\sigma},|q|^{-\sigma}\}$,
$d_l+c_k\not\in-\N_0$,
for any 
$a_i, b_j, c_k, d_l\in\CCast$ elements of 
${\bf a}, {\bf b}, {\bf c}, {\bf d}$.
Then 
\begin{eqnarray}
&&\hspace{-4.0cm}\int_{\frac{\pi}{\log q}}^{-\frac{\pi}{\log q}}
\frac{\Gamma_q({\bf d}+\sigma-ix,{\bf c}-\sigma+ix)}
{\Gamma_q({\bf b}+\sigma-ix,{\bf a}-\sigma+ix)}
q^{imx}\,\dd x=
\frac{2\pi(1-q)q^{m\sigma}}
{-\log q}{\sf A}.
\end{eqnarray}
If $d_l-d_{l'}\not\in-\N_0$, $l\ne l'$ one has
\begin{eqnarray}
&&\hspace{-0.5cm}{\sf A}=
\sum_{k=1}^B
\frac{\Gamma_q(d_k+{\bf c},{\bf d}_{[k]}-d_k)q^{md_k}}
{\Gamma_q(d_k+{\bf a},{\bf b}-d_k)}
\qhyp{A+B}{A+B-1}{q^{d_k+{\bf c}},q^{1+d_k-{\bf b}}}
{q^{d_k+{\bf a}},q^{1+d_k-{\bf d}_{[k]}}}
{q,q^{m+\Sigma b_j-\Sigma d_j}},
\end{eqnarray}
and if $c_k-c_{k'}\not\in-\N_0$, $k\ne k'$ one has
\begin{eqnarray}
&&\hspace{-0.2cm}{\sf A}=
\sum_{k=1}^A
\frac{\Gamma_q(c_k+{\bf d},{\bf c}_{[k]}-c_k)q^{-mc_k}}
{\Gamma_q(c_k+{\bf b},{\bf a}-c_k)}
\qhyp{A+B}{A+B-1}{q^{c_k+{\bf d}},q^{1+c_k-{\bf a}}}
{q^{c_k+{\bf b}},q^{1+c_k-{\bf c}_{[k]}}}
{q,q^{-m+\Sigma a_j-\Sigma c_j}},
\end{eqnarray}
where $|q^{m+\Sigma b_j-\Sigma d_j}|<1$
and $|q^{-m+\Sigma a_j-\Sigma c_j}|$ respectively.
\label{cor17}
\end{cor}
}

\medskip
{We now take the limit as $q\to 1^{-}$ and
obtain the following result.}
{
\begin{cor}
Let ${\bf a}:=\{a_1,\ldots,a_A\}$,
${\bf b}:=\{b_1,\ldots,b_B\}$,
${\bf c}:=\{c_1,\ldots,c_A\}$,
${\bf d}:=\{d_1,\ldots,d_B\}$,
$A, B\in\mathbb N_0$ (not both zero) respectively,
for any 
$a_i, b_j, c_k, d_l\in\CCast$ elements of 
${\bf a}, {\bf b}, {\bf c}, {\bf d}$,
$\Sigma a_j:=\sum_{j=1}^A a_j$,
$\Sigma b_j:=\sum_{j=1}^B b_j$,
$\Sigma c_j:=\sum_{j=1}^A c_j$,
$\Sigma d_j:=\sum_{j=1}^B d_j$,
$\sigma\in(0,\infty)$.
Define
\begin{eqnarray}
&&\hspace{-2.8cm}
{\sf B}:={\sf B}({\bf a},{\bf b},{\bf c},{\bf d}):=\frac{1}{2\pi}
\int_{-\infty}^{\infty}
\frac{\Gamma({\bf d}+\sigma-ix)\Gamma({\bf c}-\sigma+ix)}
{\Gamma({\bf b}+\sigma-ix)\Gamma({\bf a}-\sigma+ix)}
\,\dd x.
\end{eqnarray}
Then
\begin{eqnarray}
&&\hspace{-1.5cm}{\sf B}=
\sum_{k=1}^B
\frac{\Gamma(d_k+{\bf c},{\bf d}_{[k]}-d_k)}
{\Gamma(d_k+{\bf a},{\bf b}-d_k)}
\hyp{A+B}{A+B-1}{{d_k+{\bf c}},{1+d_k-{\bf b}}}
{{d_k+{\bf a}},{1+d_k-{\bf d}_{[k]}}}
{1}\\
&&\hspace{-1.0cm}=
\sum_{k=1}^A
\frac{\Gamma(c_k+{\bf d},{\bf c}_{[k]}-c_k)}
{\Gamma(c_k+{\bf b},{\bf a}-c_k)}
\hyp{A+B}{A+B-1}{{c_k+{\bf d}},{1+c_k-{\bf a}}}
{{c_k+{\bf b}},{1+c_k-{\bf c}_{[k]}}}
{1},
\end{eqnarray}
where $\Re(
\Sigma a_j+\Sigma b_j-\Sigma c_j-\Sigma d_j-1)>0$,
so that the generalized hypergeometric
series are convergent.
\end{cor}
}
{
\begin{proof}
Starting with Corollary \ref{cor17} and taking
the limit $q\to 1^{-}$ using 
\eqref{qloglim} completes the proof.
\end{proof}
}

If one can write a basic hypergeometric
function with a specific argument as a symmetric sum of two nonterminating basic hypergeometric functions with argument $q$, then
there is the following useful consequence of Theorem \ref{gascard}.

\begin{thm}\label{gascaru}
Let {$q\in\CCdag$,} ${\bf a}:=\{a_1,\ldots,a_A\}$, 
${\bf c}:=\{c_1,\ldots,c_C\}$,
be sets of {non-zero} complex numbers with cardinality 
$A,C\in\mathbb N_0$ (not {both} zero) respectively, 
${\bf d}:=\{d_1,d_2\}$, 
$c_kd_l\not\in\Omega_q$,
$z=\expe^{i\psi}$,
$\sigma\in(0,\infty)$, $d_1, d_2\in\CCast$,
such that
$|c_k|<\sigma$, $|d_1|,|d_2|<1/\sigma$, 
for any 
$c_k\in {\bf c}$.
Define
\begin{eqnarray}
&&
\hspace{-0.7cm}H({\bf a},{\bf c},{\bf d};q):=\II{d_1{;}d_2}
\frac{(d_1{\bf a};q)_\infty}
{\left(\frac{d_2}{d_1},d_1{\bf c};q\right)_{\!\infty}}
\!\qphyp{C}{A+1}{C-A-2}
{d_1{\bf c}}
{d_1{\bf a},q d_1/d_2}
{q,q}\\
&&\hspace{-0.6cm}
=\!
\frac{(d_1{\bf a};q)_\infty}
{\left(\frac{d_2}{d_1},d_1{\bf c};q\right)_{\!\infty}}
\!\qphyp{C}{A+1}{C-A-2}
{d_1{\bf c}}
{d_1{\bf a},q d_1/d_2}
{q,q}\!+\!
\frac{(d_2{\bf a};q)_\infty}
{\left(\frac{d_1}{d_2},d_2{\bf c};q\right)_{\!\infty}}
\!\qphyp{C}{A+1}{C-A-2}
{d_2{\bf c}}
{d_2{\bf a},q d_2/d_1}
{q,q}\!,\nonumber\\
\label{Hdefn}
\end{eqnarray}
where $d_l/d_{l'}\not\in\Omega_q$, $l\ne l'$, and if $C\ge A+2$,
\begin{eqnarray}
&&\hspace{-1.7cm}J({\bf a},{\bf c},{\bf d};f,q):=
\sum_{k=1}^C
\frac{
{\vartheta(fc_kd_1,\frac{f}{c_kd_2};q)}
({\bf a}/c_k
;q)_\infty}
{(c_k{\bf d},
{\bf c}_{[k]}/c_k;q)_\infty}\nonumber\\
&&\hspace{3.9cm}\times\qhyp{A+2}{C-1}
{c_k{\bf d},
qc_k/{\bf a}}{qc_k/{\bf c}_{[k]}}
{q,\frac{q(qc_k)^{C-A-2}a_1\cdots a_A}
{d_1d_2c_1\cdots c_C}},
\label{Jdefn}
\end{eqnarray}
where $c_k/c_{k'}\not\in\Omega_q$, $k\ne k'$, 
and ${}_{A+2}\phi_{C-1}$ is convergent
for $C=A+2$ if 
$|qa_1\cdots a_A|<|d_1d_2c_1\cdots c_C|$,
and is an entire function if $C>A+2$.
Then
\begin{eqnarray}
&&\hspace{-1.6cm}\int_{-\pi}^\pi
\frac{((fd_1,\frac{q}{f}d_2)
\frac{\sigma}{z}, 
(\frac{f}{d_2},\frac{q}{fd_1},{\bf a})\frac{z}{\sigma};q)_\infty}
{\left((d_1,d_2)\frac{\sigma}{z}, {\bf c}\frac{z}{\sigma};q\right)_\infty}
{\mathrm d}{\psi}\label{qqform}
=\frac{2\pi
{\vartheta(f,f\frac{d_1}{d_2};q)}
}{(q;q)_\infty}
H({\bf a},{\bf c},{\bf d};q)
\\
&&\hspace{4.96cm}=\frac{2\pi}{(q;q)_\infty}J({\bf a},{\bf c},
{\bf d};f,q),\quad (C\ge A+2),
\label{jjform}
\end{eqnarray}
{and none of the arguments of the modified theta functions are equal to some $q^m$, $m\in{\mathbb Z}$.}
\end{thm}
\begin{proof}
See proof of \cite[Theorem 2.4]{CohlCostasSantos22}.
\end{proof}

{
\begin{thm}
Let {$q\in\CCdag$,} ${\bf a}:=\{a_1,\ldots,a_A\}$, 
${\bf c}:=\{c_1,\ldots,c_C\}$,
be sets of {non-zero} complex numbers with cardinality 
$A,C\in\mathbb N_0$ (not {both} zero) respectively, 
${\bf d}:=\{d_1,d_2\}$, 
$c_k+d_l\not\in-\N_0$,
$|q|^\sigma\in(0,\infty)$, 
$q^{d_1}, q^{d_2}\in\CCast$,
such that
$|q^{c_k}|<|q|^{\sigma}$, $|q^{d_1}|,|q^{d_2}|<|q|^{-\sigma}$, 
for any 
$c_k\in {\bf c}$,
and fractional powers take their principal values.
Then
\begin{eqnarray}
&&\hspace{-1.3cm}\int_{\frac{\pi}{\log q}}^{-\frac{\pi}{\log q}}
\frac{(1\!-\!q)^{(C-A-2)(ix-\sigma)}\Gamma_q({\bf d}\!+\!\sigma\!-\!ix,
{\bf c}\!-\!\sigma\!+\!ix)}
{\Gamma_q((d_1\!+\!f,d_2\!+\!1\!-\!f)\!+\!\sigma\!-\!ix,(f\!-\!d_2,1\!-\!d_1\!-\!f,{\bf a})\!-\!\sigma\!+\!ix)}
\,\dd x\nonumber\\
&&\hspace{-0.7cm}=\frac{2\pi(1\!-\!q)}
{-\log(q)\Gamma_q(f,1\!-\!f,d_1\!-\!d_2\!+\!f,d_2\!-\!d_1\!+\!1\!-\!f)}\nonumber\\
&&\hspace{-0.0cm}\times
\II{d_1;d_2}
(1\!-\!q)^{d_1(C-A-2)}\frac{\Gamma_q({\bf c}\!+\!d_1,d_2\!-\!d_1)}
{\Gamma_q({\bf a}\!+\!d_1)}
\qphyp{C}{A+1}
{C-A-2}
{q^{{\bf c}+d_1}}
{q^{{\bf a}+d_1},
q^{1+d_1-d_2}}{q,q}\\
&&\hspace{-0.7cm}=\frac{2\pi(1\!-\!q)}{-\log q}
\sum_{k=1}^C
\frac{(1\!-\!q)^{-c_k(C-A-2)}\Gamma_q(c_k\!+\!{\bf d},{\bf c}_{[k]}\!-\!c_k)}
{\Gamma_q({\bf a}\!-\!c_k,
1\!-\!c_k\!-\!d_1\!-\!f,1\!+\!c_k\!+\!d_2\!-\!f,c_k\!+\!d_1\!+\!f,-c_k\!-\!d_2\!+\!f)}\nonumber\\
&&\hspace{0.5cm}\times
\qhyp{A+2}{C-1}{q^{c_k+{\bf d}},q^{1-{\bf a}+c_k}}
{q^{1+c_k-{\bf c}_{[k]}}}{q,
q^{1+(C-A-2)(1+c_k)+\Sigma a_j-\Sigma c_j-d_1-d_2}}.
\end{eqnarray}
\end{thm}
}
{
\begin{proof}
Using \eqref{qqform}, \eqref{jjform}, we 
respectively start along the lines
of Askey \& Roy \cite[p.~368]{AskeyRoy86}
and use the map $({\bf a},{\bf c},{\bf d},\sigma,\expe^{i\psi})\mapsto(q^{\bf a},q^{\bf c},q^{\bf d},q^\sigma,q^{ix})$
and the definition of the $q$-gamma 
function \eqref{qgam}.
This completes the proof.
\end{proof}
}

{By assuming that $C=A+2$, then the problematic $(1-q)^{C-A-2}$ terms become unity. This produces the following result.}

{
\begin{thm}
Let {$q\in\CCdag$,} ${\bf a}:=\{a_1,\ldots,a_A\}$, 
${\bf c}:=\{c_1,\ldots,c_{A+2}\}$,
be sets of {non-zero} complex numbers with cardinality 
$A\in\mathbb N_0$, 
${\bf d}:=\{d_1,d_2\}$, 
$c_k+d_l\not\in-\N_0$,
$|q|^\sigma\in(0,\infty)$, $d_1, d_2\in\CCast$,
such that
$|q^{c_k}|<|q|^{\sigma}$, $|q^{d_1}|,|q^{d_2}|<|q|^{-\sigma}$, 
for any 
$c_k\in {\bf c}$,
and fractional powers take their principal values.
Then
\begin{eqnarray}
&&\hspace{-0.8cm}\int_{\frac{\pi}{\log q}}^{-\frac{\pi}{\log q}}
\frac{
\Gamma_q({\bf d}\!+\!\sigma\!-\!ix,
{\bf c}\!-\!\sigma\!+\!ix)}
{\Gamma_q((d_1\!+\!f,d_2\!+\!1\!-\!f)\!+\!\sigma\!-\!ix,(f\!-\!d_2,1\!-\!d_1\!-\!f,{\bf a})\!-\!\sigma\!+\!ix)}
\,\dd x\nonumber\\
&&\hspace{-0.2cm}=\frac{2\pi(1\!-\!q)}
{-\log(q)\Gamma_q(f,1\!-\!f,d_1\!-\!d_2\!+\!f,d_2\!-\!d_1\!+\!1\!-\!f)}\nonumber\\
&&\hspace{3.0cm}\times
\II{d_1;d_2}
\frac{\Gamma_q({\bf c}\!+\!d_1,d_2\!-\!d_1)}
{\Gamma_q({\bf a}\!+\!d_1)}
\qhyp{A+2}{A+1}
{q^{{\bf c}+d_1}}
{q^{{\bf a}+d_1},
q^{1+d_1-d_2}}{q,q}\\
&&\hspace{-0.2cm}=\frac{2\pi(1\!-\!q)}{-\log q}
\sum_{k=1}^{C}
\frac{
\Gamma_q(c_k\!+\!{\bf d},{\bf c}_{[k]}\!-\!c_k)}
{\Gamma_q({\bf a}\!-\!c_k,
1\!-\!c_k\!-\!d_1\!-\!f,1\!+\!c_k\!+\!d_2\!-\!f,c_k\!+\!d_1\!+\!f,-c_k\!-\!d_2\!+\!f)}\nonumber\\
&&\hspace{4.5cm}\times
\qhyp{A+2}{A+1}{q^{c_k+{\bf d}},q^{1-{\bf a}+c_k}}
{q^{1+c_k-{\bf c}_{[k]}}}{q,
q^{1
+\Sigma a_j-\Sigma c_j-d_1-d_2}}.
\end{eqnarray}
\label{thm111}
\end{thm}
}

{
We now take the limit as $q\to 1^{-}$ to
obtain the following result.}

{
\begin{thm}
Let ${\bf a}:=\{a_1,\ldots,a_A\}$, 
${\bf c}:=\{c_1,\ldots,c_{A+2}\}$,
be sets of {non-zero} complex numbers with cardinality 
$A\in\mathbb N_0$, 
${\bf d}:=\{d_1,d_2\}$, 
$c_k+d_l\not\in-\N_0$,
$\sigma\in(0,\infty)$, $d_1, d_2\in\CCast$,
such that
$|c_k|<\sigma$, $|d_1|,|d_2|<1/\sigma$, 
for any 
$c_k\in {\bf c}$,
$\Re(\Sigma a_j-\Sigma c_j+1-d_1-d_2)>0$.
Then
\begin{eqnarray}
&&\hspace{-0.8cm}\int_{-\infty}^{\infty}
\frac{
\Gamma({\bf d}\!+\!\sigma\!-\!ix,
{\bf c}\!-\!\sigma\!+\!ix)\,\dd x}
{\Gamma((d_1\!+\!f,d_2\!+\!1\!-\!f)\!+\!\sigma\!-\!ix,(f\!-\!d_2,1\!-\!d_1\!-\!f,{\bf a})\!-\!\sigma\!+\!ix)}
\nonumber\\
&&\hspace{-0.7cm}=\frac{2\pi}
{\Gamma(f,1\!-\!f,d_1\!-\!d_2\!+\!f,d_2\!-\!d_2\!+\!1\!-\!f)}
\!\!\!\!
\II{d_1;d_2}
\!\!\!\!
\frac{\Gamma({\bf c}\!+\!d_1,d_2\!-\!d_1)}
{\Gamma({\bf a}\!+\!d_1)}
\!\hyp{A+2}{A+1}
{{\bf c}\!+\!d_1}
{{\bf a}\!+\!d_1,
1\!+\!d_1\!-\!d_2}{1}
\label{112a}\\
&&\hspace{-0.7cm}=2\pi
\sum_{k=1}^{C}
\frac{
\Gamma(c_k\!+\!{\bf d},{\bf c}_{[k]}\!-\!c_k)}
{\Gamma({\bf a}\!-\!c_k,
1\!-\!c_k\!-\!d_1\!-\!f,1\!+\!c_k\!+\!d_2\!-\!f,c_k\!+\!d_1\!+\!f,-c_k\!-\!d_2\!+\!f)}\nonumber\\
&&\hspace{8.2cm}\times
\hyp{A+2}{A+1}{
{c_k\!+\!{\bf d}},
{1\!-\!{\bf a}\!+\!c_k}}
{{1\!+\!c_k\!-\!{\bf c}_{[k]}}}{1}.
\label{112b}
\end{eqnarray}
\end{thm}
}
{
\begin{proof}
Start with Theorem \ref{thm111} and letting $q\to 1^{-}$ 
produces 
\eqref{112a}, \eqref{112b}.
The convergence of the generalized hypergeometric functions are unity is given by \cite[(16.2.2)]{NIST:DLMF}.
This completes the proof.
\end{proof}
}

\begin{rem}
As just indicated, it is often feasible to convert integrals of products of infinite $q$-shifted factorials to integrals of products of $q$-gamma functions. This makes a direct $q$-analogue with Mellin--Barnes integrals for the integrals in question. In some cases we have undertaken this recasting for the integrals which appear below. For instance in Corollary \ref{cor42} we recast the integral of a well-poised ${}_3\phi_2$ in terms of an integral of products of terms given by $\Gamma_q$ and $\Gamma_{q^2}$. In Theorem \ref{thm61}, we are able to write the integral of a very-well poised ${}_8W_7$ as an integral of products of terms given by $\Gamma_q$, and in this case a clear $q\to1^{-}$ limit exists and is computed. Other cases such as Theorems \ref{thm51}, \ref{thm81}, \ref{thm83} and \ref{thm94} can also be written as products of terms involving $\Gamma_q$ and $\Gamma_{q^2}$, and for Theorem \ref{thm91} it can be recast similarly as above but also including terms of the form $\Gamma_{q^3}$. However, we will leave these recastings to the reader.
\end{rem}

\section{A $q$-Mellin--Barnes integral for a ratio of theta functions}

If one would like to integrate
a ratio of an arbitrary product
of theta functions as a $q$-Mellin--Barnes integral then Theorem \ref{gascard} 
provides a powerful tool to evaluate
this integral which provides
insight into the connection
between theta functions and partial theta functions. This will
be seen in the following theorem.

\begin{thm}
Let $q\in\CCdag$, $z=\expe^{i\psi}$,
$\sigma\in(0,\infty)$,
${\bf b}\in\CCast^B$, ${\bf d}\in\CCast^D$
{with $D\ge B$}
such that $|q|/\sigma<|d_k|<1/\sigma$
for $k=1, ..., D$, and $d_l/d_k\not\in\Upsilon_q$ for 
any $d_l, d_k\in {\bf d}$ with $l\ne k$.
Then
\begin{eqnarray}
&&\hspace{-4.2cm}\int_{-\pi}^\pi \frac{\vartheta({\bf b}\frac{\sigma}{z};q)}
{\vartheta({\bf d}\frac{\sigma}{z};q)}
\expe^{im\psi}{\mathrm d}\psi=
\frac{2\pi\sigma^m}{(q;q)_\infty}
G_{m}\left(\frac{q}{\bf b},{\bf b},
\frac{q}{\bf d},{\bf d};\sigma,q\right).
\label{ptint1}
\end{eqnarray}
If $D\ge B$ then 
\vspace{-0.2cm}
\begin{eqnarray}
&&\hspace{-0.2cm}G_{m}=\frac{1}{(q;q)_\infty}
\sum_{k=1}^D\frac{\vartheta({\bf b}/d_k;q)d_k^m}
{
\vartheta({\bf d}_{[k]}/d_k;q)}
\qphyp10{D-B}{q}{-}{q,q^m(qd_k)^{D-B}\frac{b_1\cdots b_B}{d_1\cdots d_D}},
\label{ptint2}
\end{eqnarray}
which is an entire function and for $D=B$, there is a specialized sum
using the geometric series
\eqref{geom}, provided 
$|b_1\cdots b_B|<|d_1\cdots d_D|$
\vspace{-0.2cm}
\begin{eqnarray}
&&\hspace{-4.4cm}G_{m}=\frac{1}{(q;q)_\infty}
\sum_{k=1}^D\frac{\vartheta({\bf b}/d_k;q)d_k^m}
{
\vartheta({\bf d}_{[k]}/d_k;q)}
\frac{1}{1-q^m\frac{b_1\cdots b_B}{d_1\cdots d_B}}.
\label{ptint3}
\end{eqnarray}
Moreover, for $D>B$, one also has
\begin{eqnarray}
&&\hspace{0.95cm}G_m=\frac{(q^{D-B};q^{D-B})_\infty}{(q;q)_\infty}
\sum_{k=1}^D\frac{\vartheta({\bf b}/d_k;q)d_k^m}
{
\vartheta({\bf d}_{[k]}/d_k;q)}
\Theta\left(-\frac{q^m(-qd_k)^{D-B}b_1\cdots b_B}{d_1\cdots d_D};q^{D-B}\right)\!.
\label{ptint2b}
\end{eqnarray}
\end{thm}
\begin{proof}
Start with the integral on the left-hand side of \eqref{ptint1} and replace the theta function with its
definition \eqref{tfdef} in terms
of infinite $q$-shifted factorials.
Then we can easily identify
the sets
${\bf a}=q/{\bf b}$,
${\bf c}=q/{\bf d}.$
Direct substitution of these sets using Theorem \ref{gascard} provides
\eqref{ptint1}.
The utilization of \eqref{Gmt1} 
with these sets 
provides \eqref{ptint2}.
Due to the symmetric nature
of the arguments
\eqref{Gmt2} yields the same
expression in terms of nonterminating
basic hypergeometric functions.
The function which appears in the representation
of the $q$-Mellin--Barnes integral of a ratio of
theta functions
\begin{eqnarray*}
&&\hspace{-4.9cm}g_p(z;q):=\qphyp{1}{0}{D-B}{q}{-}{q,z}=\sum_{n=0}^\infty
\left((-1)^n q^{\binom{n}{2}}\right)^{D-B} z^n
\end{eqnarray*}
is connected to the partial theta function (see \S\ref{tfptf}).
The necessary relation is given by
\begin{equation}
g_p(z;q)=(q^{D-B};q^{D-B})_\infty \Theta((-1)^{D-B-1}z;q^{D-B}).
\end{equation}
Inserting this relation in 
\eqref{ptint2} yields 
\eqref{ptint2b}.
If $D=B$ then the nonterminating
basic hypergeometric series 
can be evaluated using the geometric
series \eqref{geom}.
This provides the form of
\eqref{ptint3} which 
completes the proof.
\end{proof}

\section{Symmetric representation of the Askey--Wilson moments}

\medskip
Define the Askey--Wilson weight function
\begin{equation}
w_q(x;{\bf a}|q):=
\frac{(\expe^{\pm 2i\theta};q)_\infty}
{(a\expe^{\pm i\theta},b\expe^{\pm i\theta},c\expe^{\pm i\theta},d\expe^{\pm i\theta};q)_\infty}
=\frac{\left(\pm\expe^{\pm i\theta},\pm q^\frac12\expe^{\pm i\theta};q\right)_\infty}
{\left(a\expe^{\pm i\theta},b\expe^{\pm i\theta},c\expe^{\pm i\theta},d\expe^{\pm i\theta};q\right)_\infty}
,\label{AWweight}
\end{equation}
where {${\bf a}:=\{a,b,c,d\}$}, $x=\cos\theta\in[-1,1]$, 
$a, b, c, d\not \in \Omega_q$
and we have used the identity \eqref{sq-1}.
Then, the moments of the Askey--Wilson polynomials
are given by
\begin{eqnarray}
&&\hspace{-6.2cm}\mu_n=\frac{(q,ab,\ldots,cd;q)_\infty}{4\pi(abcd;q)_\infty}
\int_{-\pi}^\pi
w_q(x;{\bf a}|q)\cos^n\theta\,
\dd\theta,
\label{moments}
\end{eqnarray}
where $\{ab,\ldots,cd\}:=\{ab,ac,ad,bc,bd,cd\}$, and 
$\mu_n$ has been normalized so that $\mu_0=1$
(see \cite[p.~170]{KimStanton2014}).

\medskip
Due to the $z=\expe^{i\theta}$ dependence
of the second equality of
the Askey--Wilson weight function
\eqref{AWweight} and
a judicious use of the binomial theorem,
the moments of the Askey--Wilson polynomials
\eqref{moments} are given by a 
$q$-Mellin--Barnes integral. Using the method of 
integral representations for nonterminating basic hypergeometric
functions (see Theorem \ref{gascard}) we are able to obtain a form symmetric
in the parameters $a$, $b$, $c$, $d$ for
the moments 
of the Askey--Wilson polynomials.

\begin{thm}
Let $n\in\N_0$, $q\in\CCdag$, $a,b,c,d\in\CCast$,
$a,b,c,d\not\in\Omega_q$. Then, the moments of
the Askey--Wilson polynomials can be given by
\begin{eqnarray}
&&\hspace{-0.3cm}\mu_n\!=\!
\frac{(ab,\ldots,cd;q)_\infty}
{2^{n+1}(abcd;q)_\infty}
\sum_{k=0}^n
\binom{n}{k}
\nonumber\\&&\hspace{2.5cm}\times
\II{a;b,c,d}
\frac{(\frac{1}{a^2};q)_\infty a^{|n-2k|}}
{(ab,ac,ad,\frac{b}{a},\frac{c}{a},\frac{d}{a};q)_\infty}
\Whyp{6}{5}{a^2}{ab,ac,ad}{q,\frac{q^{1+|n-2k|}}{abcd}}\!,
\label{AWmomen}
\end{eqnarray}
where $a/b,a/c,a/d,b/c,b/d,c/d\not\in\Upsilon_q$.
\end{thm}
\begin{proof}
Start with the left-hand side 
of \eqref{AWmomen} and take 
account of \eqref{moments}.
Applying the binomial 
theorem \cite[(1.2.2)]{NIST:DLMF} to the $\cos^n\theta$
produces
\begin{equation}
\cos^n\theta=\frac{1}{2^n}\sum_{k=0}^n
\binom{n}{k}\expe^{i\theta(n-2k)}.
\end{equation}
Applying Theorem \ref{gascard}
with $m=n-2k\in\N_0$ and
$(A,B,C,D)=(4,4,4,4)$, given by
\begin{eqnarray}
&&\hspace{-6.7cm}{\bf a}={\bf b}:=\left\{\pm 1,\pm\sqrt{q}\right\},\,
{\bf c}={\bf d}:=\left\{a,b,c,d\right\},\,
\label{acdmom}
\end{eqnarray}
produces the right-hand side of \eqref{AWmomen}
by substituting the variables in \eqref{gascardeq}.
This completes the proof.
\end{proof}

\begin{rem}
In Kim \& Stanton (2014) \cite[Theorem 2.10]{KimStanton2014},
a representation for the Askey--Wilson moments which are
symmetric in the parameters $a,b,c,d$ is given. 
Let $t\in\C$. Then 
\begin{eqnarray}
&&\hspace{-0.3cm}\mu_n=\sum_{k=0}^n (-q)^k\frac{(ta,tb,tc,td;q)_k}{(t^2,abcd;q)_k}
\Whyp87{\frac{t^2}{q}}{q^{-k},\frac{t}{a},\frac{t}{b},\frac{t}{c},\frac{t}{d}}{q,q^kabcd}\nonumber\\
&&\hspace{0.5cm}\times\sum_{s=0}^{n+1}\left(\!\binom{n}{s}\!-\!\binom{n}{s-1}\!\right)
\sum_{p=0}^{n-2s-k}
\qbinom{k\!+\!p}{k}
\qbinom{n\!-\!2s\!-\!p}{k}q^{k(2s+p-n)+\binom{k}{2}}t^{2p+2s-n},
\label{KimStan}
\end{eqnarray}
where \cite[(1.9.4)]{Koekoeketal}
\begin{equation}
\qbinom{n}{k}:=\frac{(q;q)_n}{(q;q)_k(q;q)_{n-k}}.
\end{equation}
It is interesting to note that both \eqref{KimStan} and 
\eqref{AWmomen} are symmetric in the parameters $a$, $b$, $c$,
$d$. Define ${\bf 4}:=\{1,2,3,4\}$. The representation \eqref{KimStan} is a finite sum over a lattice given by 
\begin{equation}
\{0,\ldots,n\}\times\{0,\ldots,n+1\}\times\{0,\ldots,n-2s-k\},
\end{equation} 
for each terminating very-well-poised ${}_8W_7$. On the other hand
\eqref{AWmomen} is a finite sum over a rectangular lattice 
given by
$\{0,\ldots,n\}\times{\bf 4}$
for each 
nonterminating very-well-poised ${}_6W_5$.
\end{rem}

\section{Four and five-term transformations for a nonterminating well-poised ${}_3\phi_2$}
Let $q,z\in\CCdag$, {$\tau\in(0,\infty)$,} $a,b,c,h\in\CCast$, 
$\frac{qa}{b},\frac{qa}{c}\not\in\Omega_q$,
$h,h\frac{qa}{bcz}\not\in\Upsilon_q$,
$w=\expe^{i\eta}$.
In \cite[Corollary 2.15]{CohlCostasSantos22}, we presented an integral for a nonterminating 
well-poised ${}_3\phi_2$, namely
\begin{eqnarray}
\label{int3phi2}
&&\hspace{-0.9cm}\qhyp32{a,b,c}{\frac{qa}{b},\frac{qa}{c}}{q,z}=
\frac{(q,a,\frac{qa}{bc};q)_\infty}
{2\pi\vartheta(h,h\frac{qa}{bcz};q)(\frac{qa}{b},\frac{qa}{c};q)_\infty}\nonumber\\
&&\hspace{0.5cm}\times\int_{-\pi}^\pi
\frac{((\frac{1}{h}\sqrt{\frac{bcz}{a}},h\sqrt{\frac{a}{bcz}})\frac{\tau}{w},
(qh\sqrt{\frac{a}{bcz}},\frac{q}{h}\sqrt{\frac{bcz}{a}},
q\sqrt{\frac{abz}{c}}, q\sqrt{\frac{acz}{b}},
\frac{(bcz)^\frac32}{q\sqrt{a}})\frac{w}{\tau};q)_\infty}
{((\sqrt{\frac{a}{bcz}},\sqrt{\frac{bcz}{q^2a}})\frac{\tau}{w},
(\pm\sqrt{bcz},\pm\sqrt{qbcz},q\sqrt{\frac{az}{bc}})\frac{w}{\tau};q)_\infty}\dd\eta,
\label{e41}
\end{eqnarray}
where the maximum modulus of the denominator factors in the integrand is less than unity.
This integral \eqref{int3phi2} followed from the following transformation
of a nonterminating well-poised ${}_3\phi_2$ with arbitrary argument $z\in\CCdag$, in 
terms of a sum of two nonterminating ${}_5\phi_4(q,q)$ basic hypergeometric series, 
presented
in cf.~\cite[(III.35)]{GaspRah}
\begin{eqnarray}
&&\hspace{-0.9cm}\qhyp32{a,b,c}{\frac{qa}{b},\frac{qa}{c}}{q,z}
=
\frac{(\frac{bcz}{q};q)_\infty}
{(\frac{bcz}{qa};q)_\infty}\qhyp54{\pm\sqrt{a},\pm\sqrt{qa},\frac{qa}{bc}}
{\frac{qa}{b},\frac{qa}{c},\frac{bcz}{q},\frac{q^2a}{bcz}}{q,q}
\nonumber\\&&\hspace{3.5cm}
+\frac{(a,bz,cz,\frac{qa}{bc};q)_\infty}
{(\frac{qa}{b},\frac{qa}{c},z,\frac{qa}{bcz};q)_\infty}
\qhyp54{\pm\frac{bcz}{q\sqrt{a}},\pm\frac{bcz}{\sqrt{qa}},z}
{bz,cz,\frac{bcz}{a},\frac{b^2c^2z^2}{q^2a}}{q,q}.
\label{e42}
\end{eqnarray}
\begin{rem}
\label{rem41}
In order to simplify the constraints for the 
nonterminating infinite $q$-shifted factorials, modified theta functions and nonterminating basic hypergeometric series
expressions which we will present below,
we will avoid adding the 
constraints which must occur in order to prevent
vanishing denominator factors which are not defined.
For example, 
in \eqref{e41} one must require
the constraints
\[
\frac {qa}b, \frac {qa}c\not\in\Omega_q,
\quad 
h,h\frac{qa}{bcz}\not\in\Upsilon_q.
\]
and in \eqref{e42} one must require
the constraints
\[
z,bz,cz,\frac {qa}b, \frac {qa}c,
\frac{bcz}{q},
\frac{b^2c^2z^2}{q^2a}
\not\in \Omega _q,\quad
\frac{bcz}{a}\not\in\Upsilon_q.
\]
Since it is obvious and sometimes tedious
to know for which values this happens,
we will avoid inserting such constraints in the results below.
\end{rem}

{
\begin{cor}
\label{cor42}
Let $q\in\CCdag$, 
$a,b,c,z\in\CCast$,
$|q^z|<1$.
Then
\begin{eqnarray}
&&\hspace{-0.6cm}
\qhyp32{q^a,q^b}{q^c}{q,q^z}=
\frac{-\log(q)(1\!+\!q)^{b+c+z-2\tau-\frac32}
}
{2\pi(-q,\pm q;q)_\infty (1\!-\!q)^{b+c+z-2\tau+\frac32}
}\nonumber\\
&&\hspace{0.2cm}\times\frac{\Gamma_q(f,1\!-\!f,1\!+\!f\!+\!a\!-\!b\!-\!c\!-\!z,1\!+\!b\!+\!c\!+\!z\!-\!f\!-\!1\!-\!a,1\!+\!a\!-\!b,1\!+\!a\!-\!c)}{\Gamma_q(a,1\!+\!a\!-\!b\!-\!c)}\nonumber\\
&&\hspace{0.4cm}\times
\int_{\frac{\pi}{\log q}}^{-\frac{\pi}{\log q}}
\left(\frac{1\!+\!q}{1\!-\!q}\right)^{2ix}
\frac{\Gamma_{q^2}(
(\frac{b+c+z}{2},
\frac{b+c+z+1}{2})\!+\!ix\!-\!\tau
)\Gamma_q(\frac{a-b-c-z}{2}\!+\!\tau\!-\!ix)
}
{\Gamma_q((\frac{b+c+z-a}{2}\!-\!f,f\!+\!\frac{a-b-c-z}{2})\!+\!\tau\!-\!ix,1\!+\!f\!+\!\frac{a-b-c-z}{2}\!+\!ix\!-\!\tau)}\nonumber\\
&&\hspace{1.4cm}\times\frac{\Gamma_q(\frac{b+c+z-a}{2}\!-\!1\!+\!\tau\!-\!ix,1\!+\!\frac{a+z-b-c}{2}\!+\!ix\!-\!\tau)}
{\Gamma_q((
1\!-\!f\!+\!\frac{b+c+z-a}{2},1\!+\!\frac{a+c+z-b}{2},1\!+\!\frac{a+b+z-c}{2},\frac{3(b+c+z)}{2}\!-\!\frac{a}{2}\!-\!1\!)+\!ix\!-\!\tau)}\,\dd x.
\end{eqnarray}
\end{cor}
}
{
\begin{proof}
Start with \eqref{e41} 
let $(a,b,c,z)\mapsto (q^a,q^b,q^c,q^z)$, and
use the definition of the
$q$-gamma function 
\eqref{qgam}, and
also
\eqref{a2q2} to convert the 
$(\pm a;q)_\infty$ terms.
This completes the proof.
\end{proof}
}

One can now use Theorem \ref{gascaru} to derive a five-term
representation of a nonterminating very-well-poised ${}_3\phi_2$ 
with arbitrary argument.

\begin{thm}
\label{big3phi2thm}
Let $q,z\in\CCdag$, $a,b,c,h\in\CCast$,
{and we assume there are no vanishing denominator factors
(see Remark \ref{rem41}), e.g., 
$qa/b, qa/c\not\in\Omega_q$, and $h,ha/(bcz)\not\in\Upsilon_q$.}
Then
\begin{eqnarray}
&&\hspace{-0.5cm}\qhyp32{a,b,c}{\frac{qa}{b},\frac{qa}{c}}{q,z}
\nonumber\\[0.5cm]
&&\hspace{-0.3cm}
=\frac{\vartheta(hz^{-1},h\frac{qa}{bc};q)
(a,b,c,\frac{b^2c^2z^2}{q^2a};q)_\infty}
{\vartheta(h,h\frac{qa}{bcz};q)(\frac{qa}{b},\frac{qa}{c},\frac{b^2c^2}{q^2a},z;q)_\infty}
\qhyp54
{\frac{q}{b},\frac{q}{c},\frac{qa}{bc},z,\frac{q^3a}{b^2c^2z}}
{\pm \frac{\sqrt{q^3a}}{bc},\pm \frac{q^2\sqrt{a}}{bc}}
{q,q}\nonumber\\
&&\hspace{-0.0cm}+\frac{(a,\frac{qa}{bc};q)_\infty}
{2\vartheta(h,h\frac{qa}{bcz};q)(\frac{qa}{b},\frac{qa}{c};q)_\infty}\nonumber\\
&&\hspace{0.2cm}\times\II{\pm\sqrt{a}}
\vast\llbracket
\frac
{\vartheta(h\sqrt{a},h\frac{q\sqrt{a}}{bcz};q)
(\frac{q\sqrt{a}}{b},\frac{q\sqrt{a}}{c},\frac{bcz}{q\sqrt{a}};q)_\infty}
{(\sqrt{a},\frac{q\sqrt{a}}{bc},\frac{bcz}{q\sqrt{a}};q)_\infty}
\qhyp54{\sqrt{a},\frac{b}{\sqrt{a}},\frac{c}{\sqrt{a}},
\frac{bcz}{q\sqrt{a}},\frac{q^2\sqrt{a}}{bcz}
}{-q,\pm\sqrt{q},\frac{bc}{\sqrt{a}}}{q,q}
\nonumber\\[0.1cm]
&&\hspace{1.0cm}
-\frac{\vartheta(h\sqrt{qa},h\frac{\sqrt{qa}}{bcz};q)
(\frac
{\sqrt{qa}}{b},\frac{\sqrt{qa}}{c},\frac{bcz}{\sqrt{q^3a}};q)_\infty}
{(\sqrt{qa},\frac{\sqrt{qa}}{bc},\frac{bcz}{\sqrt{qa}};q)_\infty}
\qhyp54
{\sqrt{qa},\frac{\sqrt{q}b}{\sqrt{a}},\frac{\sqrt{q}c}{\sqrt{a}},\frac{bcz}{\sqrt{qa}},
\frac{\sqrt{q^5a}}{bcz}}
{-q,\pm q^\frac32,\sqrt{q}\frac{bc}{\sqrt{a}}}{q,q}\vast\rrbracket.\nonumber\\
\end{eqnarray}
\end{thm}
\begin{proof}
In the integrand of \eqref{int3phi2}, 
the sets of parameters are given by 
$(A,C,D)=(3,5,2)$, given by
\begin{eqnarray}
&&\hspace{-2cm}{\bf a}:=\left\{q\sqrt{\frac{abz}{c}},
q\sqrt{\frac{acz}{b}},\sqrt{\frac{b^3c^3z^3}{qa}}\right\},\,
{\bf c}:=\left\{\pm\sqrt{bcz},\pm\sqrt{qbcz},q\sqrt{\frac{az}{bc}}\right\},\,
\label{acd3phi2a}
\\&&\hspace{-2cm}
{\bf d}:=\left\{\sqrt{\frac{a}{bcz}},\sqrt{\frac{bcz}{q^2a}}\right\}.
\label{acd3phi2b}
\end{eqnarray}
Now we use 
\eqref{jjform} with the sets of parameters
given in \eqref{acd3phi2a}, \eqref{acd3phi2b}
using \eqref{qqform}.
This completes the proof.
\end{proof}

If one chooses a $h=q^nz$ then since
$\vartheta(q^n;q)=0$ for all 
$n\in\Z$, 
the five-term transformation
reduces to a four-term transformation.
Then replacing the infinite $q$-shifted
factorials
with arguments involving $q^n$ and $q^{-n}$ using
\eqref{infPochdefn}, \eqref{infPoch2} produces the
following simplified result.

\begin{thm}
Let $q,z\in\CCdag$, $a,b,c\in\CCast$,
{and we assume there are no vanishing denominator factors
(see Remark \ref{rem41}).}
Then
\begin{eqnarray}
&&\hspace{-1.0cm}\qhyp32{a,b,c}{\frac{qa}{b},\frac{qa}{b}}{q,z}=
\frac{(a;q)_\infty}
{2\vartheta(z;q)(\frac{qa}{b},\frac{qa}{c},\frac{bc}{a};q)_\infty}
\nonumber\\[0.3cm]
&&\hspace{-0.3cm}\times
\II{\pm\sqrt{a}}
\vast\llbracket
\frac
{\vartheta(\sqrt{a}z;q)
(\frac{q\sqrt{a}}{b},\frac{q\sqrt{a}}{c},\frac{bc}{\sqrt{a}};q)_\infty}
{(\sqrt{a};q)_\infty}
\qhyp54{\sqrt{a},\frac{b}{\sqrt{a}},\frac{c}{\sqrt{a}},
\frac{bcz}{q\sqrt{a}},\frac{q^2\sqrt{a}}{bcz}
}{-q,\pm\sqrt{q},\frac{bc}{\sqrt{a}}}{q,q}\nonumber\\
&&\hspace{0.8cm}-
\frac{\vartheta(\sqrt{qa}z;q)
(\frac{\sqrt{qa}}{b},\frac{\sqrt{qa}}{c},
\frac{\sqrt{q}\,bc}{\sqrt{a}},\frac{bcz}{\sqrt{q^3a}};q)_\infty}
{(\sqrt{qa},\frac{bcz}{\sqrt{qa}};q)_\infty}
\qhyp54
{\sqrt{qa},\frac{\sqrt{q}b}{\sqrt{a}},\frac{\sqrt{q}c}{\sqrt{a}},
\frac{bcz}{\sqrt{qa}},\frac{\sqrt{q^5a}}{bcz}}
{-q,\pm q^\frac32,\sqrt{q}\frac{bc}{\sqrt{a}}}
{q,q}
\vast\rrbracket.
\end{eqnarray}

\end{thm}

\begin{proof}
Choose $h=q^nz$ in Theorem \ref{big3phi2thm}.
Then 
since $\vartheta(q^n;q)=0$ for all 
$n\in\Z$, 
the five-term transformation
reduces to a four-term transformation.
Then replacing the infinite $q$-shifted
factorials with arguments involving 
$q^n$ and $q^{-n}$ using
\eqref{infPochdefn}, \eqref{infPoch2}, the
factors involving $n$ all cancel,
which completes the proof.
\end{proof}

Similarly, if one chooses a $h=q^n\frac{bc}{qa}$ then the five-term transformation
reduces to a four-term transformation.

\begin{thm}
Let $q,z \in\CCdag$, $a,b,c\in\CCast$,
{and we assume there are no vanishing denominator factors
(see Remark \ref{rem41}).}
Then
\begin{eqnarray}
&&\hspace{-0.5cm}\qhyp32{a,b,c}{\frac{qa}{b},\frac{qa}{b}}{q,z}=
\frac{(a,\frac{qa}{bc};q)_\infty}
{2\vartheta(z^{-1},\frac{bc}{qa};q)(\frac{qa}{b},\frac{qa}{c};q)_\infty}
\nonumber\\[0.5cm]
&&\hspace{0cm}\times\,\,
\II{\pm\sqrt{a}}\vast\llbracket
\frac
{\vartheta(\frac{bc}{q\sqrt{a}},\frac{1}{\sqrt{a}z};q)
(\frac{q\sqrt{a}}{b},\frac{q\sqrt{a}}{c};q)_\infty}
{(\sqrt{a},\frac{q\sqrt{a}}{bc};q)_\infty}
\qhyp54{\sqrt{a},\frac{b}{\sqrt{a}},\frac{c}{\sqrt{a}},
\frac{bcz}{q\sqrt{a}},\frac{q^2\sqrt{a}}{bcz}
}{-q,\pm\sqrt{q},\frac{bc}{\sqrt{a}}}{q,q}\nonumber\\
&&\hspace{0.7cm}-
\frac{\vartheta(\frac{bc}{\sqrt{qa}},
\frac{1}{\sqrt{qa}z};q)
(\frac{\sqrt{qa}}{b},\frac{\sqrt{qa}}{c},\frac{bcz}{\sqrt{q^3a}};q)_\infty}
{(\sqrt{qa},\frac{\sqrt{qa}}{bc},\frac{bcz}{\sqrt{qa}};q)_\infty}
\qhyp54
{\sqrt{qa},\frac{\sqrt{q}\,b}{\sqrt{a}},\frac{\sqrt{q}\,c}{\sqrt{a}},
\frac{bcz}{\sqrt{qa}},\frac{\sqrt{q^5a}}{bcz}}
{-q,\pm q^\frac32,\frac{\sqrt{q}bc}{\sqrt{a}}}
{q,q}
\vast\rrbracket.
\end{eqnarray}
\end{thm}

\begin{proof}
Choose
$h=q^n\frac{bc}{qa}$ 
in Theorem \ref{big3phi2thm}.
Then 
since $\vartheta(q^n;q)=0$ for all 
$n\in\Z$, 
the five-term transformation
reduces to a four-term transformation.
Then replacing the infinite $q$-shifted
factorials with arguments involving
$q^n$ and $q^{-n}$ using
\eqref{infPochdefn}, \eqref{infPoch2}, the
factors involving $n$ all cancel,
which completes the proof.
\end{proof}

\section{Two, Four and five-term transformations
for nonterminating very-well-poised ${}_5W_4$}

{In this section} we present a $q$-Mellin--Barnes integral for a 
nonterminating very-well-poised ${}_5W_4$.

\begin{thm}
\label{thm51}
Let $q, z\in\CCdag$, {$\tau\in(0,\infty)$,} $a,b,c,h\in\CCast$,
$w=\expe^{i\eta}$,
{$h\sqrt{\frac{qa}{bcz}}\not\in\Upsilon_q$},
{and we assume there are no vanishing denominator factors
(see Remark \ref{rem41})
and the maximum modulus of the denominator factors in the integrand is less than unity.}
Then
\begin{eqnarray}
\label{int5W4}
&&\hspace{-0.9cm}\Whyp54{a}{b,c}{q,z}=
\frac{(q,qa,\frac{qa}{bc},\frac{b^2c^2z^2}{qa};q)_\infty}
{2\pi\vartheta(h,h\frac{qa}{bcz};q)(\frac{qa}{b},\frac{qa}{c},\frac{b^2c^2z^2}{a};q)_\infty}\nonumber\\
&&\hspace{0.5cm}\times\int_{-\pi}^\pi
\frac{((\frac{q}{h}\sqrt{\frac{bcz}{qa}},h\sqrt{\frac{qa}{bcz}})\frac{\tau}{w},
(h\sqrt{\frac{qa}{bcz}},\frac{q}{h}\sqrt{\frac{bcz}{qa}},
\sqrt{\frac{qacz}{b}},\sqrt{\frac{qabz}{c}},\frac{\sqrt{q}(bcz)^\frac32}{\sqrt{a}})\frac{w}{\tau};q)_\infty}
{((\sqrt{\frac{qa}{bcz}},\sqrt{\frac{bcz}{qa}})\frac{\tau}{w},
(\pm\sqrt{bcz},\pm\sqrt{qbcz},\sqrt{\frac{qaz}{bc}})\frac{w}{\tau};q)_\infty}\dd\eta.
\end{eqnarray}
\end{thm}
\begin{proof}
The integral for a nonterminating
very-well-poised ${}_5W_4$ with
arbitrary argument $z$ \eqref{int5W4}
follows from the following transformation
of a very-well-poised ${}_5W_4$ in terms
of a sum of two nonterminating balanced ${}_5\phi_4(q,q)$ basic hypergeometric 
series \cite[(3.4.4)]{GaspRah}
\begin{eqnarray}
&&\hspace{-0.3cm}\Whyp54{a}{b,c}{q,z}=\frac{(\frac{b^2c^2z^2}{qa},qbcz;q)_\infty}{(\frac{b^2c^2z^2}{a},\frac{bcz}{qa};q)_\infty}
\qhyp54{\pm\sqrt{qa},\pm q\sqrt{a},\frac{qa}{bc}}
{\frac{qa}{b},\frac{qa}{c},qbcz,\frac{q^2a}{bcz}}{q,q}
\nonumber\\
&&\hspace{5cm}+\frac{(qa,bz,cz,\frac{qa}{bc};q)_\infty}
{(\frac{qa}{b},\frac{qa}{c},z,\frac{qa}{bcz};q)_\infty}
\qhyp54{\pm\frac{bcz}{\sqrt{qa}},\pm\frac{bcz}{\sqrt{a}},z}{bz,cz,\frac{b^2c^2z^2}{a},\frac{bcz}{a}}{q,q}.
\end{eqnarray}
Now apply Theorem \ref{gascaru}
with $(A,C,D)=(3,5,2)$, given by
\begin{eqnarray}
&&\hspace{-2cm}{\bf a}:=\left\{\sqrt{\frac{qabz}{c}},
\sqrt{\frac{qacz}{b}},\sqrt{\frac{qb^3c^3z^3}{a}}\right\},\,
{\bf c}:=\left\{\pm\sqrt{bcz},\pm\sqrt{qbcz},\sqrt{\frac{qaz}{bc}}\right\},\,
\label{acd5Wa}
\\&&\hspace{-2cm}
{\bf d}:=\left\{\sqrt{\frac{qa}{bcz}},\sqrt{\frac{bcz}{qa}}\right\},
\label{acd5Wb}
\end{eqnarray}
which generates the integral in \eqref{int5W4}
using \eqref{qqform}.
{Clearly $h,h\frac{qa}{bcz}\not\in\Omega_q$ since then one would have vanishing denominator factors which are not defined. Similarly one must avoid vanishing denominator factors for other infinite $q$-shifted factorials. Furthermore, the denominator factors in the integrand must have maximum modulus less than unity so that the integral converges.}
This completes the proof.
\end{proof}

\noindent Now we apply \eqref{jjform} using the 
parameters ${\bf a}$, ${\bf c}$, ${\bf d}$ 
defined in \eqref{acd5Wa}, \eqref{acd5Wb}.
Since $C=5$, we generate a five-term transformation
for the general nonterminating very-well-poised 
${}_5W_4$. This is given in the
following theorem.

\begin{thm}
\label{thm52}
Let $q,z\in\CCdag$, $a,b,c,h\in\CCast$,
$w=\expe^{i\eta}$,
{and we assume there are no vanishing denominator factors
(see Remark \ref{rem41}).}
Then, one has the following five-term transformation
for a nonterminating very-well-poised ${}_5W_4$ 
with 
argument {$z$}:
\begin{eqnarray}
&&\hspace{-0.0cm}\Whyp54{a}{b,c}{q,z}=\frac{(qa,\frac{b^2c^2z^2}{qa},\frac{qa}{bc};q)_\infty}
{\vartheta(h,h\frac{qa}{bcz};q)(\frac{b^2c^2z^2}{a},\frac{qa}{b},\frac{qa}{c};q)_\infty}\nonumber\\
&&\hspace{0.1cm}\times\Biggl(
\frac{
\vartheta(h\sqrt{qa},h\frac{\sqrt{qa}}{bcz};q)
(\frac{\sqrt{qa}}{b},\frac{\sqrt{qa}}{c},bcz\sqrt{\frac{q}{a}};q)_\infty}
{(-1,\pm\sqrt{q},\sqrt{qa},\frac{\sqrt{qa}}{bc},\frac{bcz}{\sqrt{qa}};q)_\infty}
\qhyp54
{\sqrt{qa},\frac{qb}{\sqrt{qa}},\frac{qc}{\sqrt{qa}},\frac{bcz}{\sqrt{qa}},\frac{\sqrt{qa}}{bcz}}
{-q,\pm\sqrt{q},\frac{\sqrt{q}bc}{\sqrt{a}}}
{q,q}\nonumber\\
&&\hspace{0.3cm}+\frac{
\vartheta(-h\sqrt{qa},-h\frac{\sqrt{qa}}{bcz};q)
(-\frac{\sqrt{qa}}{b},-\frac{\sqrt{qa}}{c},-bcz\sqrt{\frac{q}{a}};q)_\infty}
{(-1,\pm\sqrt{q},-\sqrt{qa},-\frac{\sqrt{qa}}{bc},-\frac{bcz}{\sqrt{qa}};q)_\infty}
\qhyp54
{-\sqrt{qa},\frac{-\sqrt qb}{\sqrt{a}},\frac{-\sqrt qc}{\sqrt{a}},\frac{-bcz}{\sqrt{qa}},\frac{-\sqrt{qa}}{bcz}}
{-q,\pm\sqrt{q},-\frac{\sqrt{q}bc}{\sqrt{a}}}
{q,q}\nonumber\\
&&\hspace{0.3cm}+\frac{
\vartheta(hq\sqrt{a},h\frac{\sqrt{a}}{bcz};q)
(\frac{\sqrt{a}}{b},\frac{\sqrt{a}}{c},\frac{bcz}{\sqrt{a}};q)_\infty}
{(-1,\pm\frac{1}{\sqrt{q}},q\sqrt{a},\frac{\sqrt{a}}{bc},\frac{bcz}{\sqrt{a}};q)_\infty}
\qhyp54
{q\sqrt{a},\frac{qb}{\sqrt{a}},\frac{qc}{\sqrt{a}},\frac{bcz}{\sqrt{a}},\frac{q\sqrt{a}}{bcz}}
{-q,\pm q^\frac32,\frac{qbc}{\sqrt{a}}}
{q,q}\nonumber\\
&&\hspace{0.3cm}+\frac{
\vartheta(-hq\sqrt{a},-h\frac{\sqrt{a}}{bcz};q)
(\frac{-\sqrt{a}}{b},-\frac{\sqrt{a}}{c},-\frac{bcz}{\sqrt{a}};q)_\infty}
{(-1,\pm\frac{1}{\sqrt{q}},-q\sqrt{a},-\frac{\sqrt{a}}{bc},-\frac{bcz}{\sqrt{a}};q)_\infty}
\qhyp54
{-q\sqrt{a},-\frac{qb}{\sqrt{a}},-\frac{qc}{\sqrt{a}},-\frac{bcz}{\sqrt{a}},-\frac{q\sqrt{a}}{bcz}}
{-q,\pm q^\frac32,-\frac{qbc}{\sqrt{a}}}
{q,q}\nonumber\\
&&\hspace{0.3cm}+\frac
{\vartheta(\frac{h}{z},h\frac{qa}{bc};q)
(b,c,\frac{b^2c^2z}{a};q)_\infty}
{(
z,\frac{qa}{bc},\frac{b^2c^2}{qa};q)_\infty}
\qhyp54
{\frac{q}{b},\frac{q}{c},\frac{qa}{bc},z,\frac{qa}{b^2c^2z}}
{\pm\frac{q\sqrt{a}}{bc},\pm\frac{q^\frac32\sqrt{a}}{bc}}
{q,q}\Biggr).
\end{eqnarray}
\end{thm}

One can force the last term to vanish by
setting for $n\in\Z$, $h=q^nz^{-1}$ or
$h=q^{n-1}\frac{bc}{a}$, providing 
a naturally symmetric four-term transformation
for a nonterminating very-well-poised ${}_5W_4$ with 
argument $z$.

\begin{thm}Let $n\in\Z$, $q,z\in\CCdag$, 
$a,b,c\in\CCast$,
$w=\expe^{i\eta}$,
{and we assume there are no vanishing denominator factors
(see Remark \ref{rem41}).}
Then, one has the following four-term transformations
for a nonterminating very-well-poised ${}_5W_4$ with 
argument {$z$}:
\begin{eqnarray}
&&\hspace{-0.0cm}\Whyp54{a}{b,c}{q,z}=\frac{(qa,\frac{b^2c^2z^2}{qa},\frac{qa}{bc};q)_\infty}
{2\vartheta(q^nz,q^{n+1}\frac{a}{bc};q)
(-q,\frac{b^2c^2z^2}{a},\frac{qa}{b},\frac{qa}{c};q)_\infty}\nonumber\\
&&\hspace{0.1cm}\times\Biggl(
\frac{
\vartheta(q^{n+\frac12}z\sqrt{a},\frac{q^{n+\frac12}\sqrt{a}}{bc};q)
(\frac{\sqrt{qa}}{b},\frac{\sqrt{qa}}{c},bcz\sqrt{\frac{q}{a}};q)_\infty}
{(\pm\sqrt{q},\sqrt{qa},\frac{\sqrt{qa}}{bc},\frac{bcz}{\sqrt{qa}};q)_\infty}
\qhyp54
{\sqrt{qa},\frac{qb}{\sqrt{qa}},\frac{qc}{\sqrt{qa}},\frac{bcz}{\sqrt{qa}},\frac{\sqrt{qa}}{bcz}}
{-q,\pm\sqrt{q},\frac{\sqrt{q}bc}{\sqrt{a}}}
{q,q}\nonumber\\
&&\hspace{0.3cm}+\frac{
\vartheta(-q^{n+\frac12}z\sqrt{a},\frac{-q^{n+\frac12}\sqrt{a}}{bc};q)
(\frac{-\sqrt{qa}}{b},\frac{-\sqrt{qa}}{c},-bcz\sqrt{\frac{q}{a}};q)_\infty}
{(\pm\sqrt{q},-\sqrt{qa},-\frac{\sqrt{qa}}{bc},-\frac{bcz}{\sqrt{qa}};q)_\infty}
\qhyp5{4\hspace{-0.5mm}}
{\hspace{-1mm}-\sqrt{qa},\!\frac{-\sqrt qb}{\sqrt{a}},\!\frac{-\sqrt qc}{\sqrt{a}},\frac{-bcz}{\sqrt{qa}},\frac{-\sqrt{qa}}{bcz}}
{-q,\pm\sqrt{q},-\frac{\sqrt{q}bc}{\sqrt{a}}}
{\hspace{-0.5mm}q,\!q}\nonumber\\
&&\hspace{0.3cm}+\frac{
\vartheta(-q^{n+1}z\sqrt{a},-q^n\frac{\sqrt{a}}{bc};q)
(\frac{-\sqrt{a}}{b},-\frac{\sqrt{a}}{c},-\frac{bcz}{\sqrt{a}};q)_\infty}
{(\pm\frac{1}{\sqrt{q}},-q\sqrt{a},-\frac{\sqrt{a}}{bc},-\frac{bcz}{\sqrt{a}};q)_\infty}
\qhyp54
{-q\sqrt{a},-\frac{qb}{\sqrt{a}},-\frac{qc}{\sqrt{a}},-\frac{bcz}{\sqrt{a}},-\frac{q\sqrt{a}}{bcz}}
{-q,\pm q^\frac32,-\frac{qbc}{\sqrt{a}}}
{q,q}\hspace{-1.8mm}\nonumber\\
&&\hspace{0.3cm}+\frac{
\vartheta(q^{n+1}z\sqrt{a},q^n\frac{\sqrt{a}}{bc};q)
(\frac{\sqrt{a}}{b},\frac{\sqrt{a}}{c},\frac{bcz}{\sqrt{a}};q)_\infty}
{(\pm\frac{1}{\sqrt{q}},q\sqrt{a},\frac{\sqrt{a}}{bc},\frac{bcz}{\sqrt{a}};q)_\infty}
\left.\qhyp54
{q\sqrt{a},\frac{qb}{\sqrt{a}},\frac{qc}{\sqrt{a}},\frac{bcz}{\sqrt{a}},\frac{q\sqrt{a}}{bcz}}
{-q,\pm q^\frac32,\frac{qbc}{\sqrt{a}}}{q,q}\right)\\
&&\hspace{-0.0cm}=\frac{(qa,\frac{b^2c^2z^2}{qa},\frac{qa}{bc};q)_\infty}
{2\vartheta(q^{n-1}\frac{bc}{a},\frac{q^{n}}{z};q)
(-q,\frac{b^2c^2z^2}{a},\frac{qa}{b},\frac{qa}{c};q)_\infty}\nonumber\\
&&\hspace{0.1cm}\times\Biggl(
\frac{\vartheta(q^{n-\frac12}\frac{bc}{\sqrt{a}},
\frac{q^{n-\frac12}}{\sqrt{a}z};q)
(\frac{\sqrt{qa}}{b},\frac{\sqrt{qa}}{c},bcz\sqrt{\frac{q}{a}};q)_\infty}
{(\pm\sqrt{q},\sqrt{qa},\frac{\sqrt{qa}}{bc},\frac{bcz}{\sqrt{qa}};q)_\infty}
\qhyp54{\sqrt{qa},\frac{qb}{\sqrt{qa}},\frac{qc}{\sqrt{qa}},\frac{bcz}
{\sqrt{qa}},\frac{\sqrt{qa}}{bcz}}
{-q,\pm\sqrt{q},\frac{\sqrt{q}bc}{\sqrt{a}}}{q,q}\nonumber\\
&&\hspace{0.3cm}
+\frac{\vartheta(-q^{n-\frac12}\frac{bc}{\sqrt{a}},
\frac{-q^{n-\frac12}}{\sqrt{a}z};q)
(\frac{-\sqrt{qa}}{b},\frac{-\sqrt{qa}}{c},-bcz\sqrt{\frac{q}{a}};q)_\infty}
{(\pm\sqrt{q},-\sqrt{qa},\frac{-\sqrt{qa}}{bc},\frac{-bcz}{\sqrt{qa}};q)_\infty}
\qhyp5{4\hspace{-0.6mm}}{\hspace{-0.8mm}-\sqrt{qa},\frac{-\sqrt qb}{\sqrt{a}},\frac{-\sqrt qc}
{\sqrt{a}},\frac{-bcz}{\sqrt{qa}},\frac{-\sqrt{qa}}{bcz}}{-q,\pm\sqrt{q},-\frac{\sqrt{q}bc}
{\sqrt{a}}}{q,q}
\nonumber\\ &&\hspace{0.3cm}
+\frac{\vartheta(-q^{n}\frac{bc}{\sqrt{a}},-\frac{q^{n-1}}{\sqrt{a}z};q)
(\frac{-\sqrt{a}}{b},-\frac{\sqrt{a}}{c},-\frac{bcz}{\sqrt{a}};q)_\infty}
{(\pm\frac{1}{\sqrt{q}},-q\sqrt{a},-\frac{\sqrt{a}}{bc},-\frac{bcz}
{\sqrt{a}};q)_\infty}\qhyp54{-q\sqrt{a},-\frac{qb}{\sqrt{a}},-\frac{qc}
{\sqrt{a}},-\frac{bcz}{\sqrt{a}},-\frac{q\sqrt{a}}{bcz}}
{-q,\pm q^\frac32,-\frac{qbc}{\sqrt{a}}}{q,q}\nonumber\\
&&\hspace{0.3cm}
+\frac{\vartheta(q^{n}\frac{bc}{\sqrt{a}},\frac{q^{n-1}}{\sqrt{a}z};q)
(\frac{\sqrt{a}}{b},\frac{\sqrt{a}}{c},\frac{bcz}{\sqrt{a}};q)_\infty}
{(\pm\frac{1}{\sqrt{q}},q\sqrt{a},\frac{\sqrt{a}}{bc},\frac{bcz}
{\sqrt{a}};q)_\infty}\left.\qhyp54{q\sqrt{a},\frac{qb}{\sqrt{a}}
,\frac{qc}{\sqrt{a}},\frac{bcz}{\sqrt{a}},\frac{q\sqrt{a}}{bcz}}
{-q,\pm q^\frac32,\frac{qbc}{\sqrt{a}}}{q,q}\right).
\end{eqnarray}
\end{thm}
\begin{proof}
Start by inserting $h\in\{q^nz,q^{n-1}\frac{bc}{a}\}$
respectively in Theorem \ref{thm52}.
This forces the last term to vanish
producing a four-term transformation
for a nonterminating very-well-poised ${}_5W_4$ with
arbitrary argument $z$.
Note that we have used the identity $(-1;q)_\infty=2(-q;q)_\infty$.
This completes the proof.
\end{proof}

\begin{rem}
Note that one can also choose 
$
h\in \pm q^n\left\{\frac{1}{\sqrt{qa}},\frac{bcz}{\sqrt{a}},\frac{1}{q\sqrt{a}},\frac{bcz}{\sqrt{a}}\right\},$
in Theorem
\ref{thm52} with $n\in\Z$ and this will also 
produce four-term transformation formulas
for nonterminating 
very-well-poised ${}_5W_4$. However, we leave
the representation of these transformation
formulas to the reader.
\end{rem}

\section{Three and four-term transformations for the nonterminating very-well-poised ${}_8W_7$}

By starting with Bailey's transformation
of a sum of two nonterminating balanced
${}_4\phi_3$ basic hypergeometric functions
expressed as a very-well-poised ${}_8W_7$ we derive an
integral representation for the nonterminating very-well-poised ${}_8W_7$.
\begin{thm}
\label{thm61}
Let $q\in\CCdag$, $a,b,c,d,e,f,h\in\CCast$,
$\sigma\in(0,\infty)$, 
$z=\expe^{i\psi}$ such that $|q^2a^2|<|bcdef|$,
{$h\sqrt{\frac{def}{qa}}\not\in\Upsilon_q$}. Then, one has the 
following integral representation 
for a nonterminating very-well-poised ${}_8W_7$:
\begin{eqnarray}
&&\hspace{-0.8cm}\Whyp87{a}{b,c,d,e,f}{q,\frac{q^2a^2}{bcdef}}=\frac{(q,qa,\frac{qa}{bc},\frac{qa}{de},\frac{qa}{df},\frac{qa}{ef},d,e,f;q)_\infty}
{2\pi\vartheta(h,
h\frac{def}{qa}
;q)(
\frac{qa}{b},\frac{qa}{c},\frac{qa}{d},\frac{qa}{e},\frac{qa}{f};q)_\infty}\nonumber\\
&&\hspace{1.5cm}\times\int_{-\pi}^\pi
\frac{((h\sqrt{\frac{def}{qa}},\frac{q}{h}\sqrt{\frac{qa}{def}})\frac{\sigma}{z},(h\sqrt{\frac{def}{qa}},\frac{q}{h}\sqrt{\frac{qa}{def}},\frac{(qa)^\frac32}{b\sqrt{def}},\frac{(qa)^\frac32}{c\sqrt{def}})\frac{z}{\sigma};q)_\infty}
{((\sqrt{\frac{def}{qa}},\sqrt{\frac{qa}{def}})\frac{\sigma}{z},(\sqrt{\frac{qad}{ef}},\sqrt{\frac{qae}{df}},\sqrt{\frac{qaf}{de}},\frac{(qa)^\frac32}{bc\sqrt{def}})\frac{z}{\sigma};q)_\infty}
{\mathrm d}\psi,
\label{int8W7}
\end{eqnarray}
where the maximum modulus of the denominator factors in the integrand 
is less than unity and
we assume there are no vanishing denominator factors
(see Remark \ref{rem41}).
\end{thm}
\begin{proof}
We start with Bailey's transformation
of a nonterminating very-well-poised ${}_8W_7$
\cite[(17.9.16)]{NIST:DLMF}
\vspace{12pt}
\begin{eqnarray}
&&\hspace{-0.9cm}{}_8W_7\left(a;b,c,d,e,f;q,\frac{q^2a^2}{bcdef}\right)
=\frac{(\frac{qa}{de},\frac{qa}{df},\frac{qa}{ef},qa;q)_\infty}
{(\frac{qa}{def},\frac{qa}{d},\frac{qa}{e},\frac{qa}{f};q)_\infty}
\qhyp43{\frac{qa}{bc},d,e,f}{\frac{qa}{b},\frac{qa}{c},\frac{def}{a}}
{q,q}\nonumber\\&&\label{8W7to4ph3}
\hspace{3cm}+\frac{(\frac{q^2a^2}{bdef},\frac{q^2a^2}{cdef},\frac{qa}{bc},qa,d,e,f
;q)_\infty}
{(\frac{q^2a^2}{bcdef},\frac{def}{qa},\frac{qa}{b},\frac{qa}{c},\frac{qa}{d},\frac{qa}{e},\frac{qa}{f};q)_\infty}
\qhyp43{\frac{q^2a^2}{bcdef},\frac{qa}{de},\frac{qa}{df},\frac{qa}{ef}}
{\frac{q^2a^2}{bdef},\frac{q^2a^2}{cdef},\frac{q^2a}{def}}{q,q},
\end{eqnarray}
and applying Theorem \ref{gascaru} with
\begin{eqnarray}
&&\hspace{-0.7cm}{\bf a}:=\left\{\frac{(qa)^\frac32}{b\sqrt{def}},\frac{(qa)^\frac32}{c\sqrt{def}}\right\}\!,\,
{\bf c}:=\left\{\sqrt{\frac{qad}{ef}},
\sqrt{\frac{qae}{df}},\sqrt{\frac{qaf}{de}},\frac{(qa)^\frac32}{bc\sqrt{def}}\right\}\!,
\,{\bf d}:=\left\{\sqrt{\frac{def}{qa}},\sqrt{\frac{qa}{def}}\right\}\!,
\label{acd8W7}
\end{eqnarray}
completes the proof.
\end{proof}

{In the following, we will adopt Bailey's $W$ notation
for a nonterminating very-well-poised ${}_7F_6$
of argument unity (see for instance, \cite[p.~2]{Groenevelt2003})
\begin{equation}
\label{BaileyW}
W(a;b,c,d,e,f):=
\hyp76{a,\frac{a}{2}+1,b,c,d,e,f}
{\frac{a}{2},1+a\!-\!b,1+a\!-\!c,1+a\!-\!d,1+a\!-\!e,1+a\!-\!f}{1},
\end{equation}
which is absolutely convergent if 
\cite[(16.2.2)]{NIST:DLMF}
\begin{equation}
\Re(2a-(b+c+d+e+f)+2)>0.
\label{absconv}
\end{equation}
}

{
\begin{thm}
Let $a,b,c,d,e,f,\sigma\in\CC$,
$h\in\CC\setminus\Z$, $d\!+\!e\!+\!f\!-\!a\!-\!1,2\!+\!a\!-\!d\!-\!e\!-\!f,a\!-\!b\!+\!1,a\!-\!c\!+\!1,a\!-\!d\!+\!1,a\!-\!e\!+\!1,a\!-\!f\!+\!1\not\in-\N_0$.
Then
\begin{eqnarray}
&&\hspace{-0.3cm}W(a;b,c,d,e,f)=\frac{1}{2\pi}\Gamma(h,1\!-\!h,h\!+\!d\!+\!e\!+\!f\!-\!a\!-\!1,2\!+\!a\!-\!h\!-\!d\!-\!e\!-\!f)\nonumber\\
&&\hspace{0.4cm}\times
\frac{\Gamma(a\!-\!b\!+\!1,a\!-\!c\!+\!1,a\!-\!d\!+\!1,a\!-\!e\!+\!1,a\!-\!f\!+\!1)}
{\Gamma(a\!+\!1,a\!-\!b\!-\!c\!+\!1,a\!-\!d\!-\!e\!+\!1,a\!-\!d\!-\!f\!+\!1,a\!-\!e\!-\!f\!+\!1,d,e,f)}\nonumber\\
&&\hspace{1.0cm}\times \int_{-\infty}^\infty
\frac{
\Gamma(\frac{d+e+f-a-1}{2}\!+\!\sigma\!-\!ix,
\frac{a+1-d-e-f}{2}\!+\!\sigma\!-\!ix,
\frac32(a\!+\!1)\!-\!b\!-\!c\!-\!\frac{d+e+f}{2}\!-\!\sigma\!+\!ix
)}
{\Gamma(
h\!+\!\frac{d+e+f-a-1}{2}\!\mp\!\sigma\!\pm\! ix,
1\!-\!h\!+\!\frac{a+1-d-e-f}{2}\!\mp\!\sigma\!\pm\! ix
)}\nonumber\\
&&\hspace{2.3cm}\times\frac{\Gamma(\frac{a+1+f-d-e}{2}\!-\!\sigma\!+\!ix,
\frac{a+1+d-e-f}{2}\!-\!\sigma\!+\!ix,
\frac{a+1+e-d-f}{2}\!-\!\sigma\!+\!ix
)}{\Gamma(\frac32(a\!+\!1)\!-\!b\!-\!\frac{d+e+f}{2}\!+\!ix\!-\!\sigma,
\frac32(a\!+\!1)\!-\!c\!-\!\frac{d+e+f}{2}+ix\!-\!\sigma)}
\,\dd x.
\end{eqnarray}
\end{thm}
}
{
\begin{proof}
First use the map $(a,b,c,d,e,f,h,\expe^{i\psi})\mapsto(q^a,q^b,q^c,q^d,q^e,q^f,q^h,q^{ix})$
in \eqref{int8W7}. This converts 
the ${}_8W_7$ to 
\begin{eqnarray}
&&\Whyp87{q^a}{q^b,q^c,q^d,q^e,q^f}{q,q^{2a+2-b-c-d-e-f}},
\end{eqnarray}
which in the limit as $q\to 1^{-}$ becomes
$W(a;b,c,d,e,f)$.
On the right-hand side of \eqref{int8W7}, the infinite
$q$-shifted factorials can be converted
to $q$-gamma functions using 
\cite[(I.35)]{GaspRah}.
Upon taking the limit as
$q\to 1^{-}$ using 
This completes the proof.
\end{proof}
}

\noindent Now we present a theorem which 
gives a representation for a nonterminating very-well-poised ${}_8W_7$ given as a sum of four balanced ${}_4\phi_3(q)$'s.

\begin{thm}
Let $q\in\CCdag$, $a,b,c,d,e,f,h\in\CCast$
such that $|q^2a^2|<|bcdef|$,
{and we assume there are no vanishing denominator factors
(see Remark \ref{rem41}).}
Then
\begin{eqnarray}
&&\hspace{-1.6cm}\Whyp87{a}{b,c,d,e,f}{q,\frac{q^2a^2}{bcdef}}=
\frac{(qa;q)_\infty}
{\vartheta(h,h\frac{def}{qa};q)
(\frac{qa}{b},\frac{qa}{c},\frac{qa}{d},\frac{qa}{e},\frac{qa}{f};q)_\infty}\nonumber\\
&&\hspace{0.5cm}\times\Biggl(
\frac{\vartheta(h\frac{qa}{bc},h\frac{bcdef}{q^2a^2};q)(\frac{qa}{de},\frac{qa}{df},\frac{qa}{ef},b,c,d,e,f;q)_\infty}
{(\frac{q^2a^2}{bcdef},\frac{bcd}{qa},\frac{bce}{qa},\frac{bcf}{qa};q)_\infty}
\qhyp43{\frac{q^2a^2}{bcdef},\frac{qa}{bc},\frac{q}{b},
\frac{q}{c}
}{\frac{q^2a}{bcd},\frac{q^2a}{bce},\frac{q^2a}{bcf}}{q,q}\nonumber\\
&&\hspace{1.0cm}\!+\!\frac{\vartheta(hf,h\frac{de}{qa};q)(\frac{qa}{bc},\frac{qa}{bf},\frac{qa}{cf},\frac{qa}{df},\frac{qa}{ef},d,e;q)_\infty}
{(\frac{qa}{bcf},\frac{d}{f},\frac{e}{f};q)_\infty}
\qhyp43{\frac{qa}{de},\frac{bf}{a},\frac{cf}{a},f}{\frac{bcf}{a},\frac{qf}{d},\frac{qf}{e}}{q,q}
\nonumber\\
&&\hspace{1.0cm}+\frac{\vartheta(hd,h\frac{ef}{qa};q)(\frac{qa}{bc},\frac{qa}{bd},\frac{qa}{cd},\frac{qa}{de},\frac{qa}{df},e,f;q)_\infty}
{(\frac{qa}{bcd},\frac{e}{d},\frac{f}{d};q)_\infty}
\qhyp43{\frac{qa}{ef},\frac{bd}{a},\frac{cd}{a},d}{\frac{bcd}{a},\frac{qd}{e},\frac{qd}{f}}{q,q}
\nonumber\\
&&\hspace{1.0cm}+\frac{\vartheta(he,h\frac{df}{qa};q)(\frac{qa}{bc},\frac{qa}{be},\frac{qa}{ce},\frac{qa}{de},\frac{qa}{ef},d,f;q)_\infty}
{(\frac{qa}{bce},\frac{d}{e},\frac{f}{e};q)_\infty}
\qhyp43{\frac{qa}{df},\frac{be}{a},\frac{ce}{a},e}{\frac{bce}{a},\frac{qe}{d},\frac{qe}{f}}{q,q}\Biggr).
\label{8W7fourterm}
\end{eqnarray}
\end{thm}
\begin{proof}
Applying ${\bf a}$, ${\bf c}$ and ${\bf d}$ in
\eqref{acd8W7} to \eqref{jjform}
produces the result.
\end{proof}

\medskip

\begin{rem}
If one takes either $d,e,f$ equal to 
$q^{-n}$ for some $n\in\N_0$ then 
\eqref{8W7fourterm} becomes 
Watson's $q$-analogue of Whipple's 
theorem \cite[(17.9.15)]{NIST:DLMF}.
\end{rem}

\begin{rem}
Also, we were 
trying to get Bailey's transformation
of a nonterminating very-well-poised ${}_8W_7$ as a 
sum of two balanced ${}_4\phi_3$'s 
\cite[(17.9.16)]{NIST:DLMF}
as a limit case of the \eqref{8W7fourterm}, but we weren't able to.
Note also we weren't able to express
\eqref{8W7fourterm} as a sum of two nonterminating very-well-poised ${}_8W_7$'s (a la \cite[(III.37)]{GaspRah} either, although originally we had high hopes.
\end{rem}

As the parameter $h$ is free, one may choose 
$h=q^nbc/(qa)$ or $h=q^{n+2}a^2/(bcdef)$ in \eqref{8W7fourterm}. Then the first term in \eqref{8W7fourterm} vanishes
and you are left with a symmetric 
sum of three 
${}_4\phi_3$'s as a representation of a nonterminating very-well-poised ${}_8W_7$, namely the following theorem.

\begin{rem}
One might also consider the limit of
\eqref{8W7fourterm}
as $h\to\infty$. However, this limit produces a
multiplicative elliptic
function in proportion, which is doubly periodic on the entire complex plane. Therefore this limit does not exist.
\end{rem}

\begin{thm}
Let $n\in\Z$, $q\in\CCdag$, $a,b,c,d,e,f\in\CCast$
such that $|q^2a^2|<|bcdef|$,
{and we assume there are no vanishing denominator factors
(see Remark \ref{rem41}).}
Then
\begin{eqnarray}
&&\hspace{-1.0cm}\Whyp87{a}{b,c,d,e,f}{q,\frac{q^2a^2}{bcdef}}=
\frac{(qa,\frac{qa}{bc};q)_\infty}
{\vartheta(\frac{q^{n-1}bc}{a},\frac{q^{n-2}bcdef}{a^2};q)
(\frac{qa}{b},\frac{qa}{c},\frac{qa}{d},\frac{qa}{e},\frac{qa}{f};q)_\infty}\nonumber\\
&&\times\II{d;e,f} \frac{\vartheta(\frac{q^{n-1}bcd}{a},\frac{q^{n-2}bcef}{a^2};q)(\frac{qa}{bd},\frac{qa}{cd},\frac{qa}{de},\frac{qa}{df},e,f;q)_\infty}
{(\frac{qa}{bcd},\frac{e}{d},\frac{f}{d};q)_\infty}
\qhyp43{\frac{qa}{ef},\frac{bd}{a},\frac{cd}{a},d}{\frac{bcd}{a},\frac{qd}{e},\frac{qd}{f}}{q,q}.
\label{8W7threet}
\end{eqnarray}
\end{thm}

\begin{proof}
Choose $h=q^{n-1}bc/a$ in \eqref{8W7fourterm}. 
This causes the first term to vanish
and one is left with a symmetric 
sum of three nonterminating balanced ${}_4\phi_3$'s as a representation of the nonterminating very-well-poised ${}_8W_7$, namely the following theorem.
\end{proof}

Alternatively one could have chosen $h=q^2a^2/(bcdef)$ in \eqref{8W7fourterm}. This produces the following 
result.

\begin{thm}
Let $n\in\Z$, $q\in\CCdag$, $a,b,c,d,e,f\in\CCast$
such that $|q^2a^2|<|bcdef|$,
{and we assume there are no vanishing denominator factors
(see Remark \ref{rem41}).}
Then
\begin{eqnarray}
&&\hspace{-1.6cm}\Whyp87{a}{b,c,d,e,f}{q,\frac{q^2a^2}{bcdef}}=
\frac{(qa,\frac{qa}{bc};q)_\infty 
}
{\vartheta(\frac{q^{n+2}a^2}{bcdef},\frac{q^{n+1}a}{bc};q)
(\frac{qa}{b},\frac{qa}{c},\frac{qa}{d},\frac{qa}{e},\frac{qa}{f};q)_\infty}\nonumber\\
&&\times \II{d;e,f}\frac{\vartheta(\frac{q^{n+2}a^2}{bcef},\frac{q^{n+1}a}{bcd};q)(\frac{qa}{bd},\frac{qa}{cd},\frac{qa}{de},\frac{qa}{df},e,f;q)_\infty}
{(\frac{qa}{bcd},\frac{e}{d},\frac{f}{d};q)_\infty}
\qhyp43{\frac{qa}{ef},\frac{bd}{a},\frac{cd}{a},d}{\frac{bcd}{a},\frac{qd}{e},\frac{qd}{f}}{q,q}.
\label{8W7fourterm-1}
\end{eqnarray}
\end{thm}

\begin{proof}
Replace $h=q^{n+2}a^2/(bcdef)$ in \eqref{8W7fourterm} and simplifying completes the proof.
\end{proof}

\begin{rem}
Note that one can also choose for $n\in\Z$,
$h\in q^n\left\{\frac{1}d,\frac{1}e,\frac{1}f,\frac{qa}{de},\frac{qa}{df}
,\frac{qa}{ef}\right\},$
and then the four-term representation of the nonterminating very-well-poised 
${}_8W_7$ in \eqref{8W7fourterm} reduces to a three-term transformation. 
However, these representations are not symmetric and we leave their 
depictions to the reader.
\end{rem}

\section{{Summation and transformation formulas for nonterminating balanced very-well-poised ${}_8W_7$}}

By starting with Bailey's three-term transformation formula for a nonterminating very-well-poised ${}_8W_7$, we are able to prove a generalized
$q$-beta integral which will be useful 
to generate further transformation and summation formulas in the special case
where the nonterminating very-well-poised ${}_8W_7$ are also balanced.

\begin{thm}
\label{dintthm}
Let $q\in\CCdag$, $\sigma\in(0,\infty)$, $a,b,c,d,e,h\in\CCast$,
{$h, h\sqrt{\frac{a}{b}}\not\in\Upsilon_q,$}
$z=\expe^{i\psi}$. Then, one has the 
following $q$-Mellin--Barnes integral:
\begin{eqnarray}
&&\hspace{-0.5cm}\int_{-\pi}^\pi \frac{((h\sqrt{\frac{a}{b}},\frac{q}{h}\sqrt{\frac{b}{a}})\frac{\sigma}{z},(h\sqrt{\frac{a}{b}},\frac{q}{h}\sqrt{\frac{b}{a}},\pm\sqrt{b},\frac{q}{c}\sqrt{ab},\frac{q}{d}\sqrt{ab},\frac{q}{e}\sqrt{ab},\left(\frac{b}{a}\right)^\frac32cde)\frac{z}{\sigma};q)_\infty}
{((\sqrt{\frac{a}{b}},\sqrt{\frac{b}{a}})\frac{\sigma}{z},(\pm q\sqrt{b},\sqrt{ab},c\sqrt{\frac{b}{a}},d\sqrt{\frac{b}{a}},e\sqrt{\frac{b}{a}},\frac{b^\frac32}{\sqrt{a}},\frac{qa^\frac32}{cde\sqrt{b}})\frac{z}{\sigma};q)_\infty}\,\dd \psi\nonumber\\
&&\hspace{4.5cm}=2\pi\frac{\vartheta(h,h\frac{a}{b};q)(\frac{qa}{c},\frac{qa}{cd},\frac{qa}{ce},\frac{qa}{de},\frac{bcd}{a},\frac{bce}{a},\frac{bde}{a};q)_\infty}{(q,\frac{qa}{c},\frac{bc}{a},\frac{bd}{a},\frac{be}{a},b,c,d,e,\frac{qa}{cde},\frac{qa^2}{bcde};q)_\infty},
\label{defint}
\end{eqnarray}
where the maximum modulus of the denominator factors in the integrand is less than unity
and
we assume there are no vanishing denominator factors
(see Remark \ref{rem41}).
\end{thm}
\begin{proof}
Start with Bailey's three-term transformation of
a nonterminating very-well-poised ${}_8W_7$ \cite[(III.37)]{GaspRah}. Then 
as in the discussion surrounding \cite[(2.11.7)]{GaspRah}, replace $f$ using the
substitution $qa^2=bcdef$. The converts the nonterminating very-well-poised ${}_8W_7$ with argument $bd/a$ to a nonterminating very-well-poised ${}_6W_5$ with the same argument.
this nonterminating very-well-poised ${}_6W_5$ can then be summed 
using the nonterminating sum for a nonterminating very-well-poised ${}_6W_5$ \cite[(II.20)]{GaspRah}
\begin{equation}
\Whyp65{a}{b,c,d}{q,\frac{qa}{bcd}}=
\frac{\left(qa,\frac{qa}{bc},\frac{qa}{bd},\frac{qa}{cd};q\right)_\infty}
{\left(\frac{qa}{b},\frac{qa}{c},\frac{qa}{d},\frac{qa}{bcd};q\right)_\infty}.
\end{equation}
The remaining two nonterminating very-well-poised ${}_8W_7$'s both now have argument $q$ and the application of Theorem \ref{gascaru}
with $(A,C,D)=(6,8,2)$, given by
\begin{eqnarray}
&&\hspace{-5cm}{\bf a}:=\left\{\pm\sqrt{b},\frac{q}{c}\sqrt{ab},\frac{q}{d}\sqrt{ab},\frac{q}{e}\sqrt{ab},\left(\frac{b}{a}\right)^\frac32cde\right\},\\
&&\hspace{-5cm}{\bf c}:=\left\{\sqrt{ab},\pm q\sqrt{b},\frac{b^\frac32}{\sqrt{a}},c\sqrt{\frac{b}{a}},d\sqrt{\frac{b}{a}},e\sqrt{\frac{b}{a}},\frac{qa^\frac32}{cde\sqrt{b}}\right\},\\
&&\hspace{-5cm}{\bf d}:=\left\{\sqrt{\frac{a}{b}},\sqrt{\frac{b}{a}}\right\},
\end{eqnarray}
generates the integral in \eqref{defint}. This completes the proof.
\end{proof}

\noindent Now that we've generated the definite integral in 
Theorem \ref{dintthm}, we can apply \eqref{jjform}
to this definite integral to obtain 
in principle, a new
summation theorem which is an eight-term sum of nonterminating very-well-poised ${}_8W_7$'s with argument $q$.
However, when one applies \eqref{jjform}, two of the terms vanish because of leading factors of $q^{-1}$ in the list of numerator infinite $q$-shifted factorials. So we are left with a summation formula for six-terms each of which are nonterminating very-well-poised ${}_8W_7$'s with argument $q$. This is given as follows.

\begin{thm}
\label{thm72}
Let $q\in\CCdag$, $a,b,c,d,e,h\in\CCast$,
{and we assume there are no vanishing denominator factors
(see Remark \ref{rem41}).}
Then, one has the 
following six-term summation formulas for nonterminating balanced very-well-poised 
${}_8W_7$ with argument $q$:
\begin{eqnarray}
&&\frac{\vartheta(ha,\frac{h}{b};q)(\pm\frac{1}{\sqrt{a}},\frac{q}{c},\frac{q}{d},\frac{q}{e},\frac{bcde}{a^2};q)_\infty}
{(\pm\frac{q}{\sqrt{a}},a,b,\frac{b}{a},\frac{c}{a},\frac{d}{a},\frac{e}{a},\frac{qa}{bcde};q)_\infty}
\Whyp87{a}{b,c,d,e,\frac{qa^2}{bcde}}{q,q}\nonumber\\
&&\hspace{1.1cm}+\frac{\vartheta(h\frac{qa^2}{bcde},h\frac{cde}{qa};q)(\pm\frac{bcde}{qa^\frac32},\frac{bcd}{a},\frac{bce}{a},\frac{bde}{a},\frac{b^2c^2d^2e^2}{qa^3};q)_\infty}
{(\pm\frac{bcde}{a^\frac32},\frac{qa^2}{bcde},\frac{qa}{cde},\frac{bcde}{qa},\frac{b^2cde}{qa^2},\frac{bc^2de}{qa^2},\frac{bcd^2e}{qa^2},\frac{bcde^2}{qa^2};q)_\infty}\nonumber\\
&&\hspace{5cm}\times\Whyp87{\frac{q^2a^3}{b^2c^2d^2e^2}}{\frac{qa^2}{bcde},\frac{qa}{bcd},\frac{qa}{bce},\frac{qa}{bde},\frac{qa}{cde}}{q,q}\nonumber\\
&&+\II{b;c,d,e}\!\!\!
\frac{\vartheta(hb,h\frac{a}{b^2};q)(\pm\frac{\sqrt{a}}{b},\frac{qa}{bc},\frac{qa}{bd},\frac{qa}{be},\frac{cde}{a};q)_\infty}
{(\pm\frac{q\sqrt{a}}{b},b,\frac{b^2}{a},\frac{a}{b},\frac{c}{b},\frac{d}{b},\frac{e}{b},\frac{qa^2}{b^2cde};q)_\infty}
\Whyp87{\frac{b^2}{a}}{b,\frac{bc}{a},\frac{bd}{a},\frac{be}{a},\frac{qa}{cde}}{q,q}\nonumber\\
&&\hspace{1.1cm}=\frac{\vartheta(h,h\frac{a}{b};q)
(\frac{qa}{cd},\frac{qa}{ce},\frac{qa}{de},\frac{bcd}{a},\frac{bce}{a},\frac{bde}{a};q)_\infty}
{(b,c,d,e,\frac{bc}{a},\frac{bd}{a},\frac{be}{a},\frac{qa}{cde},\frac{qa^2}{bcde};q)_\infty}.
\end{eqnarray}
\end{thm}
\begin{proof}
Starting with Theorem \ref{dintthm}, we apply
\eqref{jjform}.
This produces an eight-term sum of 
nonterminating very-well-poised ${}_8W_7$'s with argument $q$.
However two of the terms vanish because of the appearance of leading factors of $q^{-1}$ in the numerator infinite $q$-shifted factorials. This leaves us with six-terms each of which are nonterminating very-well-poised ${}_8W_7$'s with argument $q$, which 
completes the proof.
\end{proof}

\begin{cor}
Let $n\in\Z$, $q\in\CCdag$, $a,b,c,d,e\in\CCast$,
{and we assume there are no vanishing denominator factors
(see Remark \ref{rem41}).}
Then, one produces the following six-term transformation formulas for nonterminating 
balanced very-well-poised ${}_8W_7$ with argument $q$:
\begin{eqnarray}
&&\hspace{-0.2cm}0=
\frac{\vartheta(q^na,\frac{q^n}{b};q)(\pm\frac{1}{\sqrt{a}},\frac{q}{c},\frac{q}{d},\frac{q}{e},\frac{bcde}{a^2};q)_\infty}
{(\pm\frac{q}{\sqrt{a}},a,b,\frac{b}{a},\frac{c}{a},\frac{d}{a},\frac{e}{a},\frac{qa}{bcde};q)_\infty}
\Whyp87{a}{b,c,d,e,\frac{qa^2}{bcde}}{q,q}\nonumber\\
&&\hspace{.4cm}+\frac{\vartheta(q^{n+1}\frac{a^2}{bcde},q^{n-1}\frac{cde}{a};q)(\pm\frac{bcde}{qa^\frac32},\frac{bcd}{a},\frac{bce}{a},\frac{bde}{a},\frac{b^2c^2d^2e^2}{qa^3};q)_\infty}
{(\pm\frac{bcde}{a^\frac32},\frac{qa^2}{bcde},\frac{qa}{cde},\frac{bcde}{qa},\frac{b^2cde}{qa^2},\frac{bc^2de}{qa^2},\frac{bcd^2e}{qa^2},\frac{bcde^2}{qa^2};q)_\infty}\nonumber\\
&&\hspace{5cm}\times
\Whyp87{\frac{q^2a^3}{b^2c^2d^2e^2}}{\frac{qa^2}{bcde},\frac{qa}{bcd},\frac{qa}{bce},\frac{qa}{bde},\frac{qa}{cde}}{q,q}\nonumber\\[-0.2cm]
&&\hspace{0.4cm}+\II{b;c,d,e}\!\!\!
\frac{\vartheta(q^nb,q^n\frac{a}{b^2};q)(\pm\frac{\sqrt{a}}{b},\frac{qa}{bc},\frac{qa}{bd},\frac{qa}{be},\frac{cde}{a};q)_\infty}
{(\pm\frac{q\sqrt{a}}{b},b,\frac{b^2}{a},\frac{a}{b},\frac{c}{b},\frac{d}{b},\frac{e}{b},\frac{qa^2}{b^2cde};q)_\infty}
\Whyp87{\frac{b^2}{a}}{b,\frac{bc}{a},\frac{bd}{a},\frac{be}{a},\frac{qa}{cde}}{q,q},
\label{f2}\\[0.8cm]
&&\hspace{-0.2cm}0=\frac{\vartheta(q^nb,\frac{q^n}{a};q)(\pm\frac{1}{\sqrt{a}},\frac{q}{c},\frac{q}{d},\frac{q}{e},\frac{bcde}{a^2};q)_\infty}
{(\pm\frac{q}{\sqrt{a}},a,b,\frac{b}{a},\frac{c}{a},\frac{d}{a},\frac{e}{a},\frac{qa}{bcde};q)_\infty}
\Whyp87{a}{b,c,d,e,\frac{qa^2}{bcde}}{q,q}\nonumber\\
&&\hspace{.4cm}+\frac{\vartheta(q^{n+1}\frac{a}{cde},q^{n-1}\frac{bcde}{a^2};q)(\pm\frac{bcde}{qa^\frac32},\frac{bcd}{a},\frac{bce}{a},\frac{bde}{a},\frac{b^2c^2d^2e^2}{qa^3};q)_\infty}
{(\pm\frac{bcde}{a^\frac32},\frac{qa^2}{bcde},\frac{qa}{cde},\frac{bcde}{qa},\frac{b^2cde}{qa^2},\frac{bc^2de}{qa^2},\frac{bcd^2e}{qa^2},\frac{bcde^2}{qa^2};q)_\infty}\nonumber\\
&&\hspace{5cm}\times
\Whyp87{\frac{q^2a^3}{b^2c^2d^2e^2}}{\frac{qa^2}{bcde},\frac{qa}{bcd},\frac{qa}{bce},\frac{qa}{bde},\frac{qa}{cde}}{q,q}\nonumber\\[-0.2cm]
&&\hspace{0.4cm}+\II{b;c,d,e}\!\!\!
\frac{\vartheta(\frac{q^nb^2}{a},\frac{q^n}{b};q)(\pm\frac{\sqrt{a}}{b},\frac{qa}{bc},\frac{qa}{bd},\frac{qa}{be},\frac{cde}{a};q)_\infty}
{(\pm\frac{q\sqrt{a}}{b},b,\frac{b^2}{a},\frac{a}{b},\frac{c}{b},\frac{d}{b},\frac{e}{b},\frac{qa^2}{b^2cde};q)_\infty}
\Whyp87{\frac{b^2}{a}}{b,\frac{bc}{a},\frac{bd}{a},\frac{be}{a},\frac{qa}{cde}}{q,q}.
\label{f4}
\end{eqnarray}
\end{cor}
\begin{proof}
Taking $h=q^n$ and $h=q^n\frac{b}{a},$ respectively in 
in Theorem \ref{thm72} produces
transformation formulas for nonterminating 
 very-well-poised ${}_8W_7$ with argument $q$
given by \eqref{f2}, \eqref{f4}.
This completes the proof.
\end{proof}

\begin{cor}
\label{thm72a}
Let $n\in\Z$, $q\in\CCdag$, $a,b,c,d,e\in\CCast$,
{and we assume there are no vanishing denominator factors
(see Remark \ref{rem41}).}
Then, one has the 
following five-term summation formulas
for nonterminating balanced very-well-poised ${}_8W_7$ with argument $q$:
\begin{eqnarray}
&&\hspace{1.1cm}
\frac{\vartheta(q^{n+1}\frac{a}{bcde},q^{n-1}\frac{cde}{a^2};q)(\pm\frac{bcde}{qa^\frac32},\frac{bcd}{a},\frac{bce}{a},\frac{bde}{a},\frac{b^2c^2d^2e^2}{qa^3};q)_\infty}
{(\pm\frac{bcde}{a^\frac32},\frac{qa^2}{bcde},\frac{qa}{cde},\frac{bcde}{qa},\frac{b^2cde}{qa^2},\frac{bc^2de}{qa^2},\frac{bcd^2e}{qa^2},\frac{bcde^2}{qa^2};q)_\infty}
\nonumber\\
&&\hspace{5cm}\times
\Whyp87{\frac{q^2a^3}{b^2c^2d^2e^2}}{\frac{qa^2}{bcde},\frac{qa}{bcd},\frac{qa}{bce},\frac{qa}{bde},\frac{qa}{cde}}{q,q}\nonumber\\
+&&\II{b;c,d,e}\!\!\!
\frac{\vartheta(q^n\frac{b}{a},\frac{q^n}{b^2};q)(\pm\frac{\sqrt{a}}{b},\frac{qa}{bc},\frac{qa}{bd},\frac{qa}{be},\frac{cde}{a};q)_\infty}
{(\pm\frac{q\sqrt{a}}{b},b,\frac{b^2}{a},\frac{a}{b},\frac{c}{b},\frac{d}{b},\frac{e}{b},\frac{qa^2}{b^2cde};q)_\infty}
\Whyp87{\frac{b^2}{a}}{b,\frac{bc}{a},\frac{bd}{a},\frac{be}{a},\frac{qa}{cde}}{q,q}\nonumber\\
&&\hspace{1.1cm}=\frac{\vartheta(\frac{q^n}{a},\frac{q^n}{b};q)
(\frac{qa}{cd},\frac{qa}{ce},\frac{qa}{de},\frac{bcd}{a},\frac{bce}{a},\frac{bde}{a};q)_\infty}
{(b,c,d,e,\frac{bc}{a},\frac{bd}{a},\frac{be}{a},\frac{qa}{cde},\frac{qa^2}{bcde};q)_\infty},\label{g1}\\
&&\hspace{1.1cm}\frac{\vartheta(q^{n+1}\frac{a^2}{cde},q^{n-1}\frac{bcde}{a};q)(\pm\frac{bcde}{qa^\frac32},\frac{bcd}{a},\frac{bce}{a},\frac{bde}{a},\frac{b^2c^2d^2e^2}{qa^3};q)_\infty}
{(\pm\frac{bcde}{a^\frac32},\frac{qa^2}{bcde},\frac{qa}{cde},\frac{bcde}{qa},\frac{b^2cde}{qa^2},\frac{bc^2de}{qa^2},\frac{bcd^2e}{qa^2},\frac{bcde^2}{qa^2};q)_\infty}\nonumber\\
&&\hspace{5cm}\times\Whyp87{\frac{q^2a^3}{b^2c^2d^2e^2}}{\frac{qa^2}{bcde},\frac{qa}{bcd},\frac{qa}{bce},\frac{qa}{bde},\frac{qa}{cde}}{q,q}\nonumber\\
+&&\II{b;c,d,e}\!\!\!
\frac{\vartheta(q^nb^2,q^n\frac{a}{b};q)(\pm\frac{\sqrt{a}}{b},\frac{qa}{bc},\frac{qa}{bd},\frac{qa}{be},\frac{cde}{a};q)_\infty}
{(\pm\frac{q\sqrt{a}}{b},b,\frac{b^2}{a},\frac{a}{b},\frac{c}{b},\frac{d}{b},\frac{e}{b},\frac{qa^2}{b^2cde};q)_\infty}
\Whyp87{\frac{b^2}{a}}{b,\frac{bc}{a},\frac{bd}{a},\frac{be}{a},\frac{qa}{cde}}{q,q}\nonumber\\
&&\hspace{1.1cm}=\frac{\vartheta(q^na,q^nb;q)
(\frac{qa}{cd},\frac{qa}{ce},\frac{qa}{de},\frac{bcd}{a},\frac{bce}{a},\frac{bde}{a};q)_\infty}
{(b,c,d,e,\frac{bc}{a},\frac{bd}{a},\frac{be}{a},\frac{qa}{cde},\frac{qa^2}{bcde};q)_\infty}
\label{g2}
\end{eqnarray}
\end{cor}
\begin{proof}
Take 
$h\in\{q^na^{-1},q^nb\},$ with $n\in\Z$, 
in Theorem \ref{thm72} 
and the first term following the symmetric sum
vanishes and one obtains five-term summation formulas
which produce \eqref{g1}, \eqref{g2} respectively. 
This completes the proof. 
\end{proof}

\begin{cor}
\label{cor75}
Let $n\in\Z$, $q\in\CCdag$, $a,b,c,d,e\in\CCast$,
{and we assume there are no vanishing denominator factors
(see Remark \ref{rem41}).}
Then, one has the 
following five-term summation formulas
for nonterminating balanced very-well-poised ${}_8W_7$ with argument $q$:
\begin{eqnarray}
&&\hspace{1.1cm}\frac{\vartheta(q^{n-1}
\frac{bcde}{a},
q^{n-1}\frac{cde}{a^2};q)(\pm\frac{1}{\sqrt{a}},\frac{q}{c},\frac{q}{d},\frac{q}{e},\frac{bcde}{a^2};q)_\infty}
{(\pm\frac{q}{\sqrt{a}},a,b,\frac{b}{a},\frac{c}{a},\frac{d}{a},\frac{e}{a},\frac{qa}{bcde};q)_\infty}
\Whyp87{a}{b,c,d,e,\frac{qa^2}{bcde}}{q,q}\nonumber\\
+&&\II{b;c,d,e}\!\!\!
\frac{\vartheta(q^{n-1}\frac{b^2cde}{a^2},q^{n-1}\frac{cde}{ab};q)(\pm\frac{\sqrt{a}}{b},\frac{qa}{bc},\frac{qa}{bd},\frac{qa}{be},\frac{cde}{a};q)_\infty}
{(\pm\frac{q\sqrt{a}}{b},b,\frac{b^2}{a},\frac{a}{b},\frac{c}{b},\frac{d}{b},\frac{e}{b},\frac{qa^2}{b^2cde};q)_\infty}
\Whyp87{\frac{b^2}{a}}{b,\frac{bc}{a},\frac{bd}{a},\frac{be}{a},\frac{qa}{cde}}{q,q}\nonumber\\
&&\hspace{1.1cm}=\frac{\vartheta(q^{n-1}\frac{bcde}{a^2},q^{n-1}\frac{cde}{a};q)
(\frac{qa}{cd},\frac{qa}{ce},\frac{qa}{de},\frac{bcd}{a},\frac{bce}{a},\frac{bde}{a};q)_\infty}
{(b,c,d,e,\frac{bc}{a},\frac{bd}{a},\frac{be}{a},\frac{qa}{cde},\frac{qa^2}{bcde};q)_\infty},
\label{h1}\\
&&\hspace{1.1cm}\frac{\vartheta(
q^{n+1}\frac{a^2}{cde},
q^{n+1}
\frac{a}{bcde}
;q)(\pm\frac{1}{\sqrt{a}},\frac{q}{c},\frac{q}{d},\frac{q}{e},\frac{bcde}{a^2};q)_\infty}
{(\pm\frac{q}{\sqrt{a}},a,b,\frac{b}{a},\frac{c}{a},\frac{d}{a},\frac{e}{a},\frac{qa}{bcde};q)_\infty}
\Whyp87{a}{b,c,d,e,\frac{qa^2}{bcde}}{q,q}\nonumber\\
+&&\II{b;c,d,e}\!\!\!
\frac{\vartheta(
q^{n+1}\frac{ab}{cde},
q^{n+1}\frac{a^2}{b^2cde};q)(\pm\frac{\sqrt{a}}{b},\frac{qa}{bc},\frac{qa}{bd},\frac{qa}{be},\frac{cde}{a};q)_\infty}
{(\pm\frac{q\sqrt{a}}{b},b,\frac{b^2}{a},\frac{a}{b},\frac{c}{b},\frac{d}{b},\frac{e}{b},\frac{qa^2}{b^2cde};q)_\infty}
\Whyp87{\frac{b^2}{a}}{b,\frac{bc}{a},\frac{bd}{a},\frac{be}{a},\frac{qa}{cde}}{q,q}\nonumber\\
&&\hspace{1.1cm}=\frac{\vartheta(
q^{n+1}\frac{a}{cde},
q^{n+1}\frac{a^2}{bcde};q)
(\frac{qa}{cd},\frac{qa}{ce},\frac{qa}{de},\frac{bcd}{a},\frac{bce}{a},\frac{bde}{a};q)_\infty}
{(b,c,d,e,\frac{bc}{a},\frac{bd}{a},\frac{be}{a},\frac{qa}{cde},\frac{qa^2}{bcde};q)_\infty}.
\label{h2}
\end{eqnarray}
\end{cor}
\begin{proof}
Take 
$h\in\{q^{n-1}\frac{bcde}{a^2},q^{n+1}\frac{a}{cde}\}$ with $n\in\Z$, 
in Theorem \ref{thm72} 
and the second term following the symmetric sum
vanishes and one obtains five-term summation formulas
which produce \eqref{h1}, \eqref{h2} respectively. 
This completes the proof. 
\end{proof}

\begin{rem}
Note that one can also choose 
$h\in q^n\left\{\frac{1}{b},\frac{1}{c},\frac{1}{d},\frac{1}{e},\frac{b^2}{a},\frac{c^2}{a},\frac{d^2}{a},\frac{e^2}{a}\right\},$
in Theorem
\ref{thm72} with $n\in\Z$ and this will also 
produce five-term summation formulas
for nonterminating balanced 
very-well-poised ${}_8W_7$. However, we leave
the representation of these summation
formulas to the reader.
\end{rem}

\section{Gasper \& Rahman's product formula for {a product of} two nonterminating ${}_2\phi_1$'s {and
for the square of a nonterminating
well-poised ${}_2\phi_1$}}

{
This section follows
from two formulas 
which can be found in Gasper \& Rahman, namely
\cite[(8.8.18)]{GaspRah}
\eqref{gasprahprod2phi1}.
\begin{eqnarray}
&&\hspace{-1.5cm}\qhyp21{a,b}{c}{q,z}
\qhyp21{a,\frac{qa}{c}}{\frac{qa}{b}}{q,z}
\nonumber\\
&&\hspace{0.5cm}=\frac{(az,\frac{abz}{c};q)_\infty}{(z,\frac{bz}{c};q)_\infty}
\qhyp65{a,\frac{c}{b},\pm\sqrt{\frac{ac}{b}},\pm\sqrt{\frac{qac}{b}}}
{c,az,\frac{qa}{b},\frac{ac}{b},\frac{qc}{bz}}{q,q}
\nonumber\\
&&\hspace{2.0cm}+\frac{(a,az,bz,\frac{c}{b},\frac{qaz}{c};q)_\infty}
{(z,z,c,\frac{qa}{b},\frac{c}{bz};q)_\infty}
\qhyp65{z,\frac{abz}{c},\pm z\sqrt{\frac{ab}{c}},\pm z\sqrt{\frac{qab}{c}}}{az,bz,\frac{qaz}{c},
\frac{abz^2}{c},
\frac{qbz}{c}
}{q,q}.
\label{gasprahprod2phi1}
\end{eqnarray}
and 
\cite[(8.8.12)]{GaspRah}
\begin{eqnarray}
&&\hspace{-1.5cm}\left(\qhyp21{a,b}{\frac{qa}{b}}{q,z}\right)^2
=\frac{(az,\frac{b^2z}{q};q)_\infty}{(z,\frac{b^2z}{qa};q)_\infty}
\qhyp54{a,\frac{qa}{b^2},\pm\frac{\sqrt{q}a}{b},-\frac{qa}{b}
}
{az,\frac{qa}{b},\frac{qa^2}{b^2},\frac{q^2a}{b^2z}}{q,q}
\nonumber\\
&&\hspace{2.5cm}+\frac{(a,az,bz,bz,\frac{qa}{b^2};q)_\infty}
{(z,z,\frac{qa}{b},\frac{qa} {b},\frac{qa}{b^2z};q)_\infty}
\qhyp54{z,-bz,\pm\frac{bz}{\sqrt{q}}, \frac{b^2z}{q}}{az,bz,
\frac{b^2z^2}{q},\frac{b^2z}{a}
}{q,q}.
\label{gasprahsq2phi1}
\end{eqnarray}
However, 
\eqref{gasprahsq2phi1}
follows
directly from their product formula
\eqref{gasprahprod2phi1}
using the substitution $c=qa/b$.
}

\subsection{{Gasper \& Rahman's product of two
nonterminating ${}_2\phi_1$'s}}

Using a product formula for a product of two nonterminating ${}_2\phi_1$'s with modulus of the argument less than unity, one can obtain a 
$q$-Mellin--Barnes integral as its representation.

\begin{thm}
\label{thm81}
Let $q,z\in\CCdag$, $\sigma\in(0,\infty)$
$a,b,c,h\in\CCast$, 
$w=\expe^{i\psi}$, {$h\sqrt{\frac{c}{bz}}\not\in\Upsilon_q$},
{and
we assume there are no vanishing denominator factors
(see Remark \ref{rem41}).}
Then, one has the 
following $q$-Mellin--Barnes integral for a 
product of two nonterminating ${}_2\phi_1$'s with arbitrary
argument $z$:
\begin{eqnarray}
&&\hspace{-0.5cm}
\int_{-\pi}^\pi 
\frac{((h\sqrt{\frac{c}{bz}},\frac{q}{h}\sqrt{\frac{bz}{c}})\frac{\sigma}{w},
(h\sqrt{\frac{c}{bz}},\frac{q}{h}\sqrt{\frac{bz}{c}},\sqrt{bcz},a\sqrt{\frac{cz}{b}},qa\sqrt{\frac{z}{bc}},a\sqrt{\frac{bz^3}{c}})\frac{w}{\sigma};q)_\infty}
{((\sqrt{\frac{c}{bz}},\sqrt{\frac{bz}{c}})\frac{\sigma}{w},(a\sqrt{\frac{bz}{c}},\sqrt{\frac{cz}{b}},\pm\sqrt{az},\pm\sqrt{qaz})\frac{w}{\sigma};q)_\infty}
\,\dd \psi\nonumber\\
&&\hspace{2.5cm}=\frac{2\pi\vartheta(h,h\frac{c}{bz};q)(z,c,\frac{qa}{b};q)_\infty}
{(q,a,\frac{c}{b},\frac{abz}{c};q)_\infty}
\qhyp21{a,b}{c}{q,z}\qhyp21{a,\frac{qa}{c}}{\frac{qa}{b}}{q,z},
\end{eqnarray}
where the maximum modulus of the denominator factors in the integrand is less than unity.
\end{thm}
\begin{proof}
Start with the formula 
for a product of two nonterminating ${}_2\phi_1$ , namely 
\eqref{gasprahprod2phi1}.
Now use Theorem \ref{gascaru} with the following 
sets of parameters with $(A,C,D)=(4,6,2)$
\begin{eqnarray}
&&{\bf a}:=\left\{\sqrt{bcz},a\sqrt{\frac{bz^3}{c}},
qa\sqrt{\frac{z}{bc}},a\sqrt{\frac{cz}{b}}\right\},\,
{\bf c}:=\left\{a\sqrt{\frac{bz}{c}},\sqrt{\frac{cz}{b}},\pm\sqrt{az},\pm\sqrt{qaz}\right\},
\label{acdproda}\\&&
{\bf d}:=\left\{\sqrt{\frac{c}{bz}},\sqrt{\frac{bz}{c}}\right\}.
\label{acdprodb}
\end{eqnarray}
This completes the proof.
\end{proof}

Now we take advantage of the $q$-Mellin--Barnes integral
for a product of two nonterminating ${}_2\phi_1$'s with arbitrary 
argument $z$ with modulus less than unity to obtain
a six-term transformation for the product of two 
nonterminating ${}_2\phi_1$'s with modulus of the argument less than unity.

\begin{thm}
Let $q,z\in\CCdag$, $a,b,c, h\in\CCast$,
$h,h\frac{c}{bz}\not\in\Upsilon_q$,
{and we assume there are no vanishing denominator factors
(see Remark \ref{rem41}).}
Then, one has the 
following six-term representation for a 
product of two nonterminating ${}_2\phi_1$'s with arbitrary
argument $z$:
\begin{eqnarray}
&&\hspace{0.0cm}\qhyp21{a,b}{c}{q,z}
\qhyp21{a,\frac{qa}{c}}{\frac{qa}{b}}{q,z}
=\frac{1}{\vartheta(h,h\frac{c}{bz};q)}
\nonumber\\
&&\hspace{0.5cm}\times\Biggl(\frac{
\vartheta(ha,h\frac{c}{abz};q)
(\frac{q}{b},\frac{c}{a},\frac{c}{b},\frac{c}{b};q)_\infty}
{(c,\frac{c}{ab},\frac{c}{ab},\frac{qa}{b};q)_\infty}
\qhyp65{a,b,\frac{qa}{c},\frac{qb}{c},\frac{abz}{c},\frac{q}{z}}
{\frac{qab}{c},\pm\sqrt{\frac{qab}{c}},\pm q\sqrt{\frac{ab}{c}}}{q,q}\nonumber\\
&&\hspace{1.5cm}+\frac{
\vartheta(h\frac{c}{b},\frac{h}{z};q)
(a,a,b,\frac{qa}{c},\frac{abz}{c},\frac{abz}{c};q)_\infty}
{(z,z,c,\frac{qa}{b},\frac{ab}{c},\frac{ab}{c};q)_\infty}
\qhyp65{\frac{q}{a},\frac{q}{b},\frac{c}{a},\frac{c}{b},z,\frac{qc}{abz}}
{\frac{qc}{ab},\pm\sqrt{\frac{qc}{ab}},\pm q\sqrt{\frac{c}{ab}}}{q,q}\nonumber\\
&&\hspace{1.5cm}+\frac{
\vartheta(h\sqrt{\frac{ac}{b}},
\frac{h}{z}\sqrt{\frac{c}{ab}},
\sqrt{\frac{bc}{a}}
;q)
(a,\frac{c}{b},
\frac{abz}{c};q)_\infty}
{(-1,\pm\sqrt{q},z,c,\frac{qa}{b},\sqrt{\frac{c}{ab}},\sqrt{\frac{ab}{c}};q)_\infty}
\nonumber\\&&\hspace{6.4cm}\times
\qhyp65{\sqrt{\frac{ac}{b}},\sqrt{\frac{bc}{a}},q\sqrt{\frac{a}{bc}},q\sqrt{\frac{b}{ac}},z\sqrt{\frac{ab}{c}},\frac{q}{z}\sqrt{\frac{c}{ab}}
}{-q,\pm\sqrt{q},q\sqrt{\frac{ab}{c}},q\sqrt{\frac{c}{ab}}}{q,q}\nonumber\\
&&\hspace{1.5cm}+\frac{
\vartheta(-h\sqrt{\frac{ac}{b}},
-\frac{h}{z}\sqrt{\frac{c}{ab}},
-\sqrt{\frac{bc}{a}}
;q)
(a,\frac{c}{b},\frac{abz}{c}
;q)_\infty}
{(-1,\pm\sqrt{q},z,c,\frac{qa}{b},-\sqrt{\frac{c}{ab}},-\sqrt{\frac{ab}{c}};q)_\infty}
\nonumber\\&&\hspace{4.7cm}\times
\qhyp65{-\sqrt{\frac{ac}{b}},-\sqrt{\frac{bc}{a}},-q\sqrt{\frac{a}{bc}},-q\sqrt{\frac{b}{ac}},-z\sqrt{\frac{ab}{c}},-\frac{q}{z}\sqrt{\frac{c}{ab}}
}{-q,\pm\sqrt{q},-q\sqrt{\frac{ab}{c}},-q\sqrt{\frac{c}{ab}}}{q,q}\nonumber\\
&&\hspace{1.5cm}+\frac{
\vartheta(h\sqrt{\frac{qac}{b}},\frac{h}{z}\sqrt{\frac{c}{qab}};q)
(a,\frac{c}{b},\frac{abz}{c},\sqrt{\frac{ac}{qb}},\sqrt{\frac{bc}{qa}},\sqrt{\frac{qa}{bc}},z\sqrt{\frac{ab}{qc}};q)_\infty}
{(-1,\pm\frac{1}{\sqrt{q}},z,c,\frac{qa}{b},\sqrt{\frac{qac}{b}},\sqrt{\frac{c}{qab}},\sqrt{\frac{ab}{qc}},z\sqrt{\frac{qab}{c}};q)_\infty}
\nonumber\\&&\hspace{6.0cm}\times
\qhyp65{\sqrt{\frac{qac}{b}},\sqrt{\frac{qbc}{a}},\sqrt{\frac{q^3a}{bc}},\sqrt{\frac{q^3b}{ac}},z\sqrt{\frac{qab}{c}},\frac{1}{z}\sqrt{\frac{q^3c}{ab}}}
{-q,\pm q^\frac32,\sqrt{\frac{q^3ab}{c}},\sqrt{\frac{q^3c}{ab}}}{q,q}\nonumber\\
&&\hspace{1.5cm}+\frac{
\vartheta(-h\sqrt{\frac{qac}{b}},-\frac{h}{z}\sqrt{\frac{c}{qab}};q)
(a,\frac{c}{b},\frac{abz}{c},-\sqrt{\frac{ac}{qb}},-\sqrt{\frac{bc}{qa}},-\sqrt{\frac{qa}{bc}},-z\sqrt{\frac{ab}{qc}};q)_\infty}
{(-1,\pm\frac{1}{\sqrt{q}},z,c,\frac{qa}{b},-\sqrt{\frac{qac}{b}},-\sqrt{\frac{c}{qab}},-\sqrt{\frac{ab}{qc}},-z\sqrt{\frac{qab}{c}};q)_\infty}
\nonumber\\&&\hspace{4.0cm}\times
\qhyp65{-\sqrt{\frac{qac}{b}},-\sqrt{\frac{qbc}{a}},-\sqrt{\frac{q^3a}{bc}},-\sqrt{\frac{q^3b}{ac}},-z\sqrt{\frac{qab}{c}},-\frac{1}{z}\sqrt{\frac{q^3c}{ab}},}
{-q,\pm q^\frac32,-\sqrt{\frac{q^3ab}{c}},-\sqrt{\frac{q^3c}{ab}}}{q,q}
\vast{)}.\nonumber\\
\end{eqnarray}

\end{thm}
\begin{proof}
Applying ${\bf a}$, ${\bf c}$ and ${\bf d}$ in
\eqref{acdproda}, \eqref{acdprodb} to \eqref{jjform}
and connecting with Theorem \ref{thm81}
completes the proof.
\end{proof}

\subsection{{Gasper \& Rahman's formula for the square of a nonterminating well-poised ${}_2\phi_1$}}

Using the formula for the square of a
nonterminating well-poised ${}_2\phi_1$ with modulus of the argument less than unity one can obtain a 
$q$-Mellin--Barnes integral as its representation.

\begin{thm}
\label{thm83}
Let $q,z\in\CCdag$, $a,b,h\in\CCast$, 
$\sigma\in(0,\infty)$,
{$\frac{h}{b}\sqrt{\frac{qa}{z}}\not\in\Upsilon_q$},
$w=\expe^{i\psi}$. Then, one has the 
following $q$-Mellin--Barnes integral for the
square of a nonterminating well-poised 
${}_2\phi_1$ with arbitrary
argument $z$:
\begin{eqnarray}
&&\hspace{-0.5cm}
\int_{-\pi}^\pi 
\frac{((\frac{h}{b}\sqrt{\frac{qa}{z}},\frac{qb}{h}\sqrt{\frac{z}{qa}})\frac{\sigma}{w},
(\frac{h}{b}\sqrt{\frac{qa}{z}},\frac{qb}{h}\sqrt{\frac{z}{qa}},\sqrt{qaz},b\sqrt{\frac{az^3}{q}},\frac{1}{b}\sqrt{qa^3z})\frac{w}{\sigma};q)_\infty}
{((b\sqrt{\frac{z}{qa}},\frac{1}{b}\sqrt{\frac{qa}{z}})\frac{\sigma}{w},(b\sqrt{\frac{az}{q}},\pm\sqrt{az},\frac{1}{b}\sqrt{qaz},-\sqrt{qaz})\frac{w}{\sigma};q)_\infty}
\,\dd \psi\nonumber\\
&&\hspace{2.5cm}=\frac{2\pi\vartheta(h,h\frac{qa}{b^2z};q)(z,\frac{qa}{b},\frac{qa}{b};q)_\infty}
{(q,a,\frac{qa}{b^2},\frac{b^2z}{q};q)_\infty}
\left(\qhyp21{a,b}{\frac{qa}{b}}{q,z}\right)^2,
\end{eqnarray}
where the maximum modulus of the denominator factors in the integrand is less than unity
and 
we assume there are no vanishing denominator factors
(see Remark \ref{rem41}).
\end{thm}
\begin{proof}
Start with the formula
for {the square of a nonterminating well-poised} ${}_2\phi_1$, namely 
\eqref{gasprahsq2phi1}.
Now use Theorem \ref{gascaru} with the following 
sets of parameters with $(A,C,D)=(3,5,2)$
\begin{eqnarray}
&&{\bf a}:=\left\{\sqrt{qaz},
b\sqrt{\frac{az^3}{q}},
\frac{\sqrt{qa^3z}}{b}
\right\},\,
{\bf c}:=\left\{b\sqrt{\frac{az}{q}},\frac{\sqrt{qaz}}{b},\pm\sqrt{az},-\sqrt{qaz}\right\},
\label{acdprodAa}\\&&
{\bf d}:=\left\{\frac{1}{b}\sqrt{\frac{qa}{z}},b\sqrt{\frac{z}{qa}}\right\}.
\label{acdprodAb}
\end{eqnarray}
This completes the proof.
\end{proof}

Now we take advantage of the $q$-Mellin--Barnes integral
for the square of a nonterminating well-poised ${}_2\phi_1$ with arbitrary 
argument $z$ with modulus less than unity to obtain
a five-term transformation for the square.

\begin{thm}
Let $q,z\in\CCdag$, $a,b,h\in\CCast$, 
{and we assume there are no vanishing denominator factors
(see Remark \ref{rem41}).}
Then, one has the 
following five-term representation for a 
square of a nonterminating well-poised ${}_2\phi_1$ with arbitrary
argument $z$:
\begin{eqnarray}
&&\hspace{0.0cm}\left(\qhyp21{a,b}{\frac{qa}{b}}{q,z}\right)^2
=\frac{1}{\vartheta(h,h\frac{qa}{b^2z};q)}
\nonumber\\
&&\hspace{0.5cm}\times\Biggl(\frac{
\vartheta(ha,h\frac{q}{b^2z};q)
(\frac{q}{b},\frac{q}{b},\frac{qa}{b^2},\frac{qa}{b^2};q)_\infty}
{(\frac{qa}{b},\frac{qa}{b},\frac{q}{b^2},\frac{q}{b^2};q)_\infty}
\qhyp54{a,b,\frac{b^2}{a},\frac{b^2z}{q},\frac{q}{z}}
{b^2,\pm b\sqrt{q},-b}{q,q}\nonumber\\
&&\hspace{1.0cm}+\frac{
\vartheta(h\frac{qa}{b^2},\frac{h}{z};q)
(a,a,b,b,\frac{b^2z}{q},\frac{b^2z}{q};q)_\infty}
{(z,z,\frac{b^2}{q},\frac{b^2}{q},\frac{qa}{b},\frac{qa}{b};q)_\infty}
\qhyp54{\frac{q}{a},\frac{q}{b},\frac{qa}{b^2},z,\frac{q^2}{b^2z}}
{\frac{q^2}{b^2},\pm\frac{q^\frac32}{b},-\frac{q}{b}}{q,q}\nonumber\\
&&\hspace{1.0cm}+\frac{
\vartheta(h\frac{a\sqrt{q}}{b},
h\frac{\sqrt{q}}{bz}
;q)
(\sqrt{q},a,\frac{qa}{b^2},
\frac{b^2z}{q};q)_\infty}
{(-1,-\sqrt{q},z,\frac{qa}{b},\frac{qa}{b},\frac{b}{\sqrt{q}},\frac{\sqrt{q}}{b};q)_\infty}
\qhyp54{\sqrt{q},\frac{a\sqrt{q}}{b},
\frac{b\sqrt{q}}{a},\frac{bz}{\sqrt{q}},
\frac{q^\frac32}{bz}
}{-q,-\sqrt{q},b\sqrt{q},\frac{q^\frac32}{b}}{q,q}\nonumber\\
&&\hspace{1.0cm}+\frac{
\vartheta(-h\frac{a\sqrt{q}}{b},
-h\frac{\sqrt{q}}{bz}
;q)
(-\sqrt{q},a,\frac{qa}{b^2},
\frac{b^2z}{q}
;q)_\infty}
{(-1,\sqrt{q},z,\frac{qa}{b},\frac{qa}{b},-\frac{b}{\sqrt{q}},-\frac{\sqrt{q}}{b};q)_\infty}
\qhyp54{-\sqrt{q},-\frac{a\sqrt{q}}{b},
-\frac{b\sqrt{q}}{a},-\frac{bz}{\sqrt{q}},
-\frac{q^\frac32}{bz}
}{-q,\sqrt{q},-b\sqrt{q},-\frac{q^\frac32}{b}}{q,q}\nonumber\\
&&\hspace{1.0cm}+\frac{
\vartheta(-h\frac{qa}{b},-\frac{h}{bz};q)
(-1,a,-\frac{a}{b},\frac{qa}{b^2},-\frac{bz}{q},\frac{b^2z}{q};q)_\infty}
{(z,-bz,\pm\frac{1}{\sqrt{q}},\frac{qa}{b},\frac{qa}{b},-\frac{qa}{b},-\frac{1}{b},-\frac{b}{q};q)_\infty}
\qhyp54{-q,-bz,-\frac{qa}{b},-\frac{qb}{a},-\frac{q^2}{bz}}
{-qb,\pm q^\frac32,-\frac{q^2}{b}}{q,q}.
\label{jj2phi1sq}
\end{eqnarray}
\end{thm}
\begin{proof}
Applying ${\bf a}$, ${\bf c}$ and ${\bf d}$ in
\eqref{acdprodAa}, \eqref{acdprodAb} to \eqref{jjform}
and connecting with Theorem \ref{thm81}
completes the proof.
\end{proof}

\begin{rem}
If you choose for some $n\in\Z$, 
$
h\in q^n\left\{\frac{1}{a},\frac{b^2z}{q},\frac{b^2}{qa},z,\pm\frac{b}{a\sqrt{q}},\pm\frac{bz}{\sqrt{q}},-\frac{b}{qa},-bz\right\},$
then the five-term transformation formula for the square of a nonterminating well-poised ${}_2\phi_1$ with
arbitrary argument $z$ in 
\eqref{jj2phi1sq} reduces to a four-term transformation formula.
However we leave the representation of 
these transformation formulas to the reader.
\end{rem}

An interesting application of this expansions
\eqref{gasprahsq2phi1},
\eqref{jj2phi1sq} are given 
in the following corollary which 
gives nonterminating four-term and two-term summation theorems for sums of nonterminating ${}_4\phi_3$'s.

\begin{cor}
Let $q\in\CCdag$, $a,b,h\in\CCast$,
{and we assume there are no vanishing denominator factors
(see Remark \ref{rem41}).}
Then, one has the 
following analogues of the
Bailey-Daum $q$-Kummer sum in the specialization as $z=-q/b$ for the square of the nonterminating well-poised ${}_2\phi_1$:
\begin{eqnarray}
&&\hspace{-1.0cm}\frac{(-q,-q;q)_\infty(qa,qa,\frac{q^2a}{b^2}\frac{q^2a}{b^2};q^2)_\infty}
{(-\frac{q}{b},-\frac{q}{b},\frac{qa}{b},\frac{qa}{b};q)_\infty}\nonumber\\
&&\hspace{0.0cm}=\frac{(-b,-\frac{qa}{b};q)_\infty}{(-\frac{q}{b},-\frac{b}{a};q)_\infty}\qhyp43{\pm\frac{a\sqrt{q}}{b},\frac{qa}{b^2},a}{\pm\frac{qa}{b},\frac{qa^2}{b^2}}{q,q}
\nonumber\\&&
\hspace{1cm}+
\frac{(-q,-q,a,-\frac{qa}{b},\frac{qa}{b^2};q)_\infty}{(\frac{qa}{b},\frac{qa}{b},-\frac{q}{b},-\frac{q}{b},-\frac{a}{b};q)_\infty}\qhyp43{\pm\sqrt{q},-b,-\frac{q}{b}}{-q,-\frac{qa}{b},-\frac{qb}{a}}{q,q}
\label{BDS1}\\
&&\hspace{0.0cm}=
\frac{1}{\vartheta(h,-h\frac{a}{b};q)}\Biggl(
\frac{\vartheta(ha,-\frac{h}{b};q)(\frac{q}{b},\frac{q}{b},\frac{qa}{b^2},\frac{qa}{b^2};q)_\infty}{(\frac{qa}{b},\frac{qa}{b},\frac{q}{b^2},\frac{q}{b^2};q)_\infty}\qhyp43{a,\pm b,\frac{b^2}{a}}{b^2,\pm b\sqrt{q}}{q,q}\nonumber\\
&&\hspace{1.0cm}+\frac{\vartheta(h\frac{qa}{b^2},-h\frac{b}{q};q)(a,a,\pm b,\pm b;q)_\infty}{(-\frac{q}{b},-\frac{q}{b},\frac{b^2}{q},\frac{b^2}{q},\frac{qa}{b},\frac{qa}{b};q)_\infty}
\qhyp43{\frac{qa}{b^2},\pm\frac{q}{b},\frac{q}{a}}{\frac{q^2}{b^2},\pm\frac{q^\frac32}{b}}{q,q}\nonumber\\
&&\hspace{1.0cm}+\frac{\vartheta(h\frac{a\sqrt{q}}{b},-\frac{h}{\sqrt{q}};q)(\sqrt{q},a,-b,\frac{qa}{b^2};q)_\infty}{(-1,-\sqrt{q},-\frac{q}{b},\frac{qa}{b},\frac{qa}{b},\frac{b}{\sqrt{q}},\frac{\sqrt{q}}{b};q)_\infty}
\qhyp43{\frac{a\sqrt{q}}{b},\frac{b\sqrt{q}}{a},\pm\sqrt{q}}{-q,b\sqrt{q},\frac{q^\frac32}{b}}{q,q}\nonumber\\
&&\hspace{1.0cm}+\frac{\vartheta(-h\frac{a\sqrt{q}}{b},\frac{h}{\sqrt{q}};q)(-\sqrt{q},a,-b,\frac{qa}{b^2};q)_\infty}{(-1,\sqrt{q},-\frac{q}{b},\frac{qa}{b},\frac{qa}{b},-\frac{b}{\sqrt{q}},-\frac{\sqrt{q}}{b};q)_\infty}
\qhyp43{-\frac{a\sqrt{q}}{b},-\frac{b\sqrt{q}}{a},\pm\sqrt{q}}{-q,-b\sqrt{q},-\frac{q^\frac32}{b}}{q,q}\Biggr).
\label{BDS2}
\end{eqnarray}
\end{cor}
\begin{proof}
Simply start with 
\eqref{gasprahsq2phi1}, 
\eqref{jj2phi1sq}, let $z=-q/b$
and then compare the resulting
expressions to the Bailey-Daum $q$-Kummer 
sum \cite[(17.6.5)]{NIST:DLMF}
\begin{equation}
\qhyp21{a,b}{\frac{qa}{b}}{q,-\frac{q}{b}}
=\frac{(-q;q)_\infty(qa,\frac{q^2a}{b^2};q^2)_\infty}{(-\frac{q}{b},\frac{qa}{b};q)_\infty},
\end{equation}
where $|q|<b$. For the expression 
which arises from \eqref{jj2phi1sq}
one of the term vanishes due to the appearance
of a unity factor in one of the 
numerator infinity $q$-shifted 
factorials.
This completes the proof.
\end{proof}

\begin{rem}
If you choose for some $n\in\Z$, 
$
h\in q^n\left\{\frac{1}{a},-b,\frac{b^2}{qa},-\frac{q}{b},\pm\frac{b}{a\sqrt{q}},\pm\sqrt{q}\right\},$
then the four-term summation theorem in 
\eqref{BDS2} reduces to a three-term summation theorem.
However we leave the representation of 
these summation theorems to the reader.
\end{rem}

\begin{rem}
We also note that using the following
transformation of a nonterminating well-poised
${}_2\phi_1$ to a nonterminating very-well-poised
${}_8W_7$
cf.~\cite[(8.8.16)]{GaspRah}
\begin{equation}
\qhyp21{a,b}{\frac{qa}{b}}{q,z}
=\frac{(\pm z\sqrt{a},\pm zb\sqrt{\frac{a}{q}};q)_\infty}
{(z,-az,\pm\frac{bz}{\sqrt{q}};q)_\infty}
\Whyp87{-\frac{az}{q}}{-\frac{bz}{q},\pm\sqrt{a},\pm\frac{\sqrt{qa}}{b}}{q,-bz},
\end{equation}
where we assume there are no vanishing denominator factors
(see Remark \ref{rem41}).
Hence, the above formulas can also be expressed as a product of two nonterminating very-well-poised ${}_8W_7$'s. However, we 
leave the representations of these formulas
to the reader. Furthermore since the
nonterminating 
${}_2\phi_1$ can also be expressed
in terms of a nonterminating ${}_2\phi_2$ 
\cite[(17.9.1)]{NIST:DLMF}, and as well as a sum of two ${}_3\phi_2$'s with vanishing 
numerator parameter \cite[(17.9.3)]{NIST:DLMF} or as a sum of 
two ${}_3\phi_2$'s with vanishing denominator
parameter \cite[(17.9.3\_5)]{NIST:DLMF}, there
are many alternative ways to represent the above formulas as well as
in the previous section.
\end{rem}

\section{{Verma \& Jain's transformations for a very-well-poised nonterminating ${}_{12}W_{11}$ and ${}_{10}W_9$}}

In a paper by Verma \& Jain (1982) \cite{VermaJain82}, the authors
present examples of transformation formulas for very-well-poised
basic hypergeometric series. In this section we exploit several
of these formulas to derive integral representations for these
nonterminating very-well-poised basic hypergeometric series and
then use the integral representations to derive new transformations
for these nonterminating very-well-poised basic hypergeometric
series. We will focus in particular on a formula they derived
for very-well-poised ${}_{12}W_{11}$ and ${}_{10}W_9$. 

\subsection{The Verma--Jain ${}_{12}W_{11}$ transformation}

Using a transformation for a nonterminating very-well-poised
${}_{12}W_{11}$ as a sum of two nonterminating balanced ${}_6\phi_5$ with argument $q$ we derive the following
integral representation.

\medskip
Define the notation $\om a:=\{a,\omega a,\omega^2 a\}$
for $a\in\CC$, $\omega=\expe^{\frac23 \pi i}$, the cube root of unity.
Note that
$\omega^3=1$, 
$\omega^5=\omega^{-1}=\omega^2$, 
$\omega^4=\omega^{-2}=\omega$. 

\begin{thm}
\label{thm91}
Let $q\in\CCdag$, $\omega=\expe^{\frac13(2\pi i)}$, $\sigma,h,a,x,y,z\in\CCast$. Then
\begin{eqnarray}
&&\hspace{-1.0cm}\Whyp{12}{11}{a}{x,qx,q^2x,y,qy,q^2y,z,qz,q^2z}{q^3,\frac{q^3a^4}{(xyz)^2}}\nonumber\\
&&\hspace{-0.0cm}=
\frac{(q,qa,\frac{qa}{xy},\frac{qa}{xz},\frac{qa}{yz},
x,y,z,
\om a^\frac13
;q)_\infty}
{2\pi\vartheta(h,h\frac{xyz}{qa};q)(\frac{qa}{x},\frac{qa}{y},\frac{qa}{z},a;q)_\infty}\nonumber\\
&&\hspace{1.0cm}\times
\int_{-\pi}^\pi
\frac{((h\sqrt{\frac{xyz}{qa}},\frac{q}{h}
\sqrt{\frac{qa}{xyz}})\frac{\sigma}{w},
(h\sqrt{\frac{xyz}{qa}},\frac{q}{h}\sqrt{\frac{xya}{xyz}},\pm a\sqrt{\frac{q}{xyz}},\pm\frac{qa}{\sqrt{xyz}})\frac{w}{\sigma};q)_\infty}
{((\sqrt{\frac{xyz}{qa}},\sqrt{\frac{qa}{xyz}})\frac{\sigma}{w},(
\sqrt{\frac{qax}{yz}},
\sqrt{\frac{qay}{xz}},
\sqrt{\frac{qaz}{xy}},
\om 
a^\frac56 
\sqrt{\frac{q}{xyz}}
)\frac{w}{\sigma};q)_\infty},
\end{eqnarray}
where the maximum modulus of the denominator factors in the integrand is less than unity
and 
we assume there are no vanishing denominator factors
(see Remark \ref{rem41}).
\end{thm}
\begin{proof}
Start with the transformation formula for a nonterminating
very-well-poised ${}_{12}W_{11}$ in terms of a sum of
two balanced ${}_6\phi_5$ with argument $q$
\cite[6.1]{VermaJain82}:
\begin{eqnarray}
&&\hspace{-1.4cm}\Whyp{12}{11}{a}{x,qx,q^2x,y,qy,q^2y,z,qz,q^2z}{q^3,\frac{q^3a^4}{(xyz)^3}}\nonumber\\
&&\hspace{-0.4cm}=\frac{(qa,\frac{qa}{xy},\frac{qa}{xz},\frac{qa}{yz};q)_\infty}
{(\frac{qa}{x},\frac{qa}{y},\frac{qa}{z},\frac{qa}{xyz};q)_\infty}
\qhyp65{x,y,z,
\om a^\frac13
}
{\frac{xyz}{a},\pm\sqrt{a},\pm\sqrt{qa}}{q,q}\nonumber\\
&&\hspace{0.3cm}+\frac{(x,y,z,\frac{q^2a^3}{(xyz)^2};q)_\infty}
{(\frac{qa}{x},\frac{qa}{y},\frac{qa}{z},\frac{xyz}{qa};q)_\infty}\
\frac{(q^3a;q^3)_\infty}{(\frac{q^3a^4}{(xyz)^3};q^3)_\infty}
\qhyp65{\frac{qa}{xy},\frac{qa}{xz},\frac{qa}{yz},
\om\frac{qa^\frac43}{xyz}
}
{\frac{q^2a}{xyz},\pm\frac{qa^\frac32}{xyz},\pm\frac{(qa)^\frac32}{xyz}}{q,q}.
\end{eqnarray}
Now use Theorem \ref{gascaru} with the following 
sets of parameters with $(A,C,D)=(3,5,2)$
\begin{eqnarray}
&&\hspace{-0.5cm}{\bf a}:=\left\{\pm \frac{q^\frac12a}{\sqrt{xyz}},
\pm\frac{qa}{\sqrt{xyz}}
\right\},\,
{\bf c}:=\left\{
\sqrt{\frac{qax}{yz}},
\sqrt{\frac{qay}{xz}},
\sqrt{\frac{qaz}{xy}},
\om\frac{q^\frac12a^\frac56}{\sqrt{xyz}}
\right\},
\label{acdprodAaa}\\&&
\hspace{-0.5cm}{\bf d}:=\left\{\sqrt{\frac{xyz}{qa}},\sqrt{\frac{qa}{xyz}}\right\}.
\label{acdprodAaab}
\end{eqnarray}
This completes the proof.
\end{proof}

Now we compute a six-term transformation for the ${}_{12}W_{11}$.

\begin{thm}
Let $q\in\CCdag$, $\omega=\expe^{\frac23\pi i}$, $a,x,y,z\in\CCast$,
and we assume there are no vanishing denominator factors
(see Remark \ref{rem41}). Then
\begin{eqnarray}
&&\hspace{-0.9cm}\Whyp{12}{11}{a}{x,qx,q^2x,y,qy,q^2y,z,qz,q^2z}{q^3,\frac{q^3a^4}{(xyz)^3}}=\frac{(qa,x,y,z,\frac{qa}{xy},\frac{qa}{xz},\frac{qa}{yz},
\om a^\frac13;q)_\infty}
{\vartheta(h,h\frac{xyz}{qa};q)(\frac{qa}{x},\frac{qa}{y},\frac{qa}{z},a;q)_\infty}
\nonumber\\
&&\hspace{-0.7cm}\times\Vast(
\,\,\II{x;y,z}\frac{\vartheta(hx,h\frac{yz}{qa};q)(\frac{a}{x^2},y,z;q)_\infty}
{(x,\frac{y}{x},\frac{z}{x},\frac{qa}{yz},\om\frac{a^{1/3}}{x};q)_\infty}\qhyp65{x,\frac{qa}{yz},\pm\frac{q^\frac12 x}{\sqrt{a}},\pm\frac{qx}{\sqrt{a}}}{\frac{qx}{y},\frac{qx}{z},\om\frac{qx}{a^{1/3}}}{q,q}\nonumber\\
&&\hspace{-0.4cm}+\frac{1}{(\omega,\omega^2;q)_\infty}\!\!\!\!\!\!
\II{\omega;\omega^2,\omega^3}\!\!\!\!\!\!\frac{\vartheta(h\omega a^\frac13,h\frac{\omega^2xyz}{qa^{4/3}};q)}
{(\frac{q\omega a^{4/3}}{xyz},
\frac{\omega^2x}{a^{1/3}},
\frac{\omega^2y}{a^{1/3}},
\frac{\omega^2z}{a^{1/3}};q)_\infty}
\qhyp65{\omega a^\frac13,\frac{q\omega a^{4/3}}{xyz},
\pm\frac{\omega q^\frac12}{a^{1/6}},
\pm\frac{\omega q}{a^{1/6}}}
{q\omega,q\omega^2,\frac{q\omega a^{1/3}}{x},
\frac{q\omega a^{1/3}}{y},
\frac{q\omega a^{1/3}}{z}}{q,q}\!\!\vast).
\label{eqthm92}
\end{eqnarray}
\end{thm}
\begin{proof}
Applying ${\bf a}$, ${\bf c}$ and ${\bf d}$ in
\eqref{acdprodAaa}, \eqref{acdprodAaab} to \eqref{jjform}
and connecting with Theorem \ref{thm91}
completes the proof.
\end{proof}

\begin{rem}
If you choose for some $n\in\Z$, 
$h\in q^n\left\{\frac{1}{x},\frac{1}{y},\frac{1}{z},\frac{qa}{xy},\frac{qa}{xz},\frac{qa}{yz},
\frac{\om}{a^{1/3}},
\frac{\om qa^\frac43}{xyz}
\right\},$
then the six-term transformation in 
\eqref{eqthm92} reduces to a five-term transformation. 
However we leave the representation of 
these transformations to the reader.
\end{rem}

\subsection{The Verma--Jain ${}_{10}W_{9}$ transformation}

There is a transformation for a ${}_{10}W_9$ represented as a sum of two balanced ${}_5\phi_4(q)$ which we originally located in 
\cite[(7.11)]{RahmanVerma93},
where it refers the reader to
\cite[(4.1)]{VermaJain82}. 
In this subsection we exploit
this transformation to write the ${}_{10}W_9$ as a $q$-Mellin--Barnes integral and then from there to obtain five and four term transformation formulas 
for the ${}_{10}W_9$.

\begin{thm}
\label{thm94}
Let $q\in\CCdag$, $a,b,x,y,z,\sigma\in\CCast$, $w=\expe^{i\psi}$. Then
\begin{eqnarray}
&&\hspace{-0.5cm}\Whyp{10}{9}{a^2}{b^2,x,qx,y,qy,z,qz}
{q^2,\frac{q^3a^6}{(bxyz)^2}}=
\frac{(q,\pm qa,x,y,z,\frac{qa^2}{xy},\frac{qa^2}{xz},\frac{qa^2}{yz};q)_\infty}
{2\pi\vartheta(h,h\frac{xyz}{qa^2};q)(\pm\frac{qa}{b},\frac{qa^2}{x},\frac{qa^2}{y},\frac{qa^2}{z};q)_\infty}\nonumber\\
&&\hspace{1.0cm}\times
\int_{-\pi}^\pi
\frac{((\frac{h}{a}\sqrt{\frac{xyz}{q}},\frac{qa}{h}\sqrt{\frac{q}{xyz}}
)\frac{\sigma}{w},(\frac{h}{a}\sqrt{\frac{xyz}{q}},
\frac{qa}{h}\sqrt{\frac{q}{xyz}},\pm\frac{qa^2}{\sqrt{xyz}},
\frac{q^\frac32a^3}{b^2\sqrt{xyz}})\frac{w}{\sigma};q)_\infty}
{((\frac{1}{a}\sqrt{\frac{xyz}{q}},a\sqrt{\frac{q}{xyz}}
)\frac{\sigma}{w},(a\sqrt{\frac{qx}{yz}},
a\sqrt{\frac{qy}{xz}},
a\sqrt{\frac{qz}{xy}},
\pm\frac{qa^2}{b\sqrt{xyz}})\frac{w}{\sigma}q)_\infty}\dd\psi,
\end{eqnarray}
where the maximum modulus of the denominator factors in the integrand is less than unity
and 
we assume there are no vanishing denominator factors
(see Remark \ref{rem41}).
\end{thm}
\begin{proof}
Start with \cite[(4.1)]{VermaJain82}
\begin{eqnarray}
&&\hspace{-1.3cm}\Whyp{10}{9}{a^2}{b^2,x,qx,y,qy,z,qz}{q^2,
\frac{q^3a^6}{(bxyz)^2}}\nonumber\\
&&\hspace{0.0cm}=\frac{(qa^2,\frac{qa^2}{xy},\frac{qa^2}{xz},\frac{qa^2}{yz};q)_\infty}
{(\frac{qa^2}{x},\frac{qa^2}{y},\frac{qa^2}{z},
\frac{qa^2}{xyz};q)_\infty}\qhyp54{x,y,z,\pm\sqrt{q}\frac{a}{b}}
{q\frac{a^2}{b^2},\frac{xyz}{a^2},\pm\sqrt{q} a}{q,q}\nonumber\\
&&\hspace{0.3cm}+
\frac{(qa^2,x,y,z,\frac{q^2a^4}{b^2xyz},\pm\sqrt{q}\frac{a}{b},\pm\frac{q^\frac32a^2}{xyz};q)_\infty}{(\frac{qa^2}{b^2},\frac{qa^2}{x},\frac{qa^2}{y},\frac{qa^2}{z},\frac{xyz}{qa^2},\pm\sqrt{q}a,\pm\frac{q^\frac32a^3}{bxyz};q)_\infty}
\qhyp54{\frac{qa^2}{xy},\frac{qa^2}{xz},\frac{qa^2}{yz},\pm\frac{q^\frac32a^3}{bxyz}}
{\frac{q^2a^2}{xyz},\frac{q^2a^4}{b^2xyz},\pm
\frac{q^\frac32a^3}{xyz}}{q,q},
\end{eqnarray}
where we have replaced $(a,b)\mapsto(a^2,b^2)$ in the original reference.
Note that in \cite[(4.1)]{VermaJain82} there is a typo in 
the second term, namely the numerator factor $q^2a^2/(bx^2yz)$ should
be replaced with $q^2a^2/(bxyz)$. Furthermore note that this same formula appears also in \cite[(7.11)]{RahmanVerma93}, however there are several typos in the second term of their formula.
Now use Theorem \ref{gascaru} with the following 
sets of parameters with $(A,C,D)=(3,5,2)$
\begin{eqnarray}
&&\hspace{-0.5cm}{\bf a}:=\left\{
\pm\frac{qa^2}{\sqrt{xyz}},
\frac{q^\frac32 a^3}{b^2\sqrt{xyz}}
\right\},\,
{\bf c}:=\left\{
a\sqrt{\frac{qx}{yz}},
a\sqrt{\frac{qy}{xz}},
a\sqrt{\frac{qz}{xy}},
\pm\frac{qa^2}{b\sqrt{xyz}}
\right\},
\label{acdprodAaaa}\\&&
\hspace{-0.5cm}{\bf d}:=\left\{\frac{1}{a}\sqrt{\frac{xyz}{q}},a\sqrt{\frac{q}{xyz}}\right\}.
\label{acdprodAaaab}
\end{eqnarray}
This completes the proof.
\end{proof}

Now we compute a five-term transformation for the ${}_{10}W_{9}$.

\begin{thm}
Let $q\in\CCdag$, 
$a,b,x,y,z\in\CCast$,
and we assume there are no vanishing denominator factors
(see Remark \ref{rem41}). Then
\begin{eqnarray}
&&\hspace{-1.2cm}\Whyp{10}{9}{a^2}{b^2,x,qx,y,qy,z,qz}{q^2,\frac{q^3a^6}{(bxyz)^2}}=\frac{(\pm qa,x,y,z,\frac{qa^2}{xy},\frac{qa^2}{xz},\frac{qa^2}{yz};q)_\infty}
{\vartheta(h,h\frac{xyz}{qa^2};q)
(\pm\frac{qa}{b},\frac{qa^2}{x},\frac{qa^2}{y},\frac{qa^2}{z};q)_\infty}
\nonumber\\
&&\hspace{-0.7cm}\times\Vast(
\,\,\II{x;y,z}\frac{\vartheta(hx,h\frac{yz}{qa^2};q)(\pm\frac{\sqrt{q}a}{x},\frac{qa^2}{b^2x};q)_\infty}
{(x,\frac{y}{x},\frac{z}{x},\frac{qa^2}{yz},\pm\frac{\sqrt{q}a}{bx};q)_\infty}
\qhyp54{x,\frac{qa^2}{yz},\pm\frac{q^\frac12 x}{a},\frac{b^2x}{a^2}}
{\frac{qx}{y},\frac{qx}{z},\pm\frac{\sqrt{q}bx}{a}}{q,q}\nonumber\\
&&\hspace{-0.0cm}+\frac{(\pm b;q)_\infty}{2(-q;q)_\infty}\vast(\frac{\vartheta(h\frac{\sqrt{q}a}{b},h\frac{bxyz}{q^{3/2}a^3};q)}
{(\frac{q^{3/2}a^3}{bxyz},
\frac{bx}{\sqrt{q}a},
\frac{by}{\sqrt{q}a},
\frac{bz}{\sqrt{q}a}
;q)_\infty}
\qhyp54{\pm\frac{q}{b},\frac{\sqrt{q}a}{b},
\frac{\sqrt{q}b}{a},\frac{q^{3/2}a^3}{bxyz}}
{-q,\frac{q^{3/2}a}{bx},
\frac{q^{3/2}a}{by},
\frac{q^{3/2}a}{bz}
}{q,q}\nonumber\\
&&\hspace{0.4cm}+\frac{\vartheta(-h\frac{\sqrt{q}a}{b},-h\frac{bxyz}{q^{3/2}a^3};q)}
{(-\frac{q^{3/2}a^3}{bxyz},
-\frac{bx}{\sqrt{q}a},
-\frac{by}{\sqrt{q}a},
-\frac{bz}{\sqrt{q}a}
;q)_\infty}
\qhyp54{\pm\frac{q}{b},-\frac{\sqrt{q}a}{b},
-\frac{\sqrt{q}b}{a},-\frac{q^{3/2}a^3}{bxyz}}
{-q,-\frac{q^{3/2}a}{bx},
-\frac{q^{3/2}a}{by},
-\frac{q^{3/2}a}{bz}
}{q,q}\vast)
\!\!\vast).
\label{eq95}
\end{eqnarray}
\end{thm}
\begin{proof}
Applying ${\bf a}$, ${\bf c}$ and ${\bf d}$ in
\eqref{acdprodAaaa}, \eqref{acdprodAaaab} to \eqref{jjform}
and connecting with Theorem \ref{thm94}
completes the proof.
\end{proof}

\begin{rem}
If you choose for some $n\in\Z$, 
$h\in q^n\left\{\frac{1}{x},\frac{1}{y},\frac{1}{z},\frac{qa^2}{xy},\frac{qa^2}{xz},\frac{qa^2}{yz},
\pm\frac{b}{\sqrt{q}a},
\pm\frac{q^\frac32 a^3}{bxyz}
\right\},
$
then the five-term transformation in 
\eqref{eq95} reduces to a four-term transformation. 
However we leave the representation of 
these transformations to the reader.
\end{rem}

\section*{Acknowledgements}
Much appreciation to Jimmy Mc Laughlin
for introducing the first author to the
literature on partial theta functions and to 
Tom Koornwinder for valuable discussions.
R.S.C-S acknowledges financial support through the research project PGC2018–096504-B-C33 supported by Agencia Estatal de Investigaci\'on of Spain.


\def\cprime{$'$} \def\dbar{\leavevmode\hbox to 0pt{\hskip.2ex \accent"16\hss}d}

\end{document}